\documentclass[a4paper, 12pt]{article}

\usepackage[english]{babel} 
\usepackage[ansinew]{inputenc}
\usepackage[LGR,T1]{fontenc}

\usepackage{verbatim}
\usepackage{color}
\usepackage{calc}
\usepackage{hyperref}
\usepackage[dvips]{graphicx}
\hypersetup{colorlinks=true, linkcolor=blue, filecolor=blue, urlcolor=blue}

\usepackage[mathscr]{eucal}
\usepackage{amsthm,amscd,amsmath,amsfonts,amssymb,textcomp}
\usepackage[all]{xy}
\everymath{\displaystyle}

\newcommand{\mf}[1]{\mathfrak{#1}}
\newcommand{\mr}[1]{\mathrm{#1}}

\newcommand{\MBC}{\mathbb{C}}

\newcommand{\MBF}{\mathbb{F}}

\newcommand{\MBN}{\mathbb{N}}

\newcommand{\MBQ}{\mathbb{Q}}

\newcommand{\MBZ}{\mathbb{Z}}

\newcommand{\MCD}{\mathcal{D}}

\newcommand{\MCF}{\mathcal{F}}

\newcommand{\MCI}{\mathcal{I}}
\newcommand{\MCJ}{\mathcal{J}}

\newcommand{\MCO}{\mathcal{O}}

\newcommand{\MCR}{\mathcal{R}}

\newcommand{\MCU}{\mathcal{U}}

\newcommand{\MCZ}{\mathcal{Z}}

\newcommand{\MFa}{\mathfrak{a}}

\newcommand{\MFb}{\mathfrak{b}}

\newcommand{\MFc}{\mathfrak{c}}

\newcommand{\MFd}{\mathfrak{d}}
\newcommand{\MFD}{\mathfrak{D}}

\newcommand{\MFf}{\mathfrak{f}}

\newcommand{\MFg}{\mathfrak{g}}

\newcommand{\MFm}{\mathfrak{m}}

\newcommand{\MFn}{\mathfrak{n}}

\newcommand{\MFp}{\mathfrak{p}}
\newcommand{\MFP}{\mathfrak{P}}
\newcommand{\MFq}{\mathfrak{q}}

\newcommand{\SCs}{\textsc{s}}

\newcommand{\GGd}{\delta}
\newcommand{\GGD}{\Delta}

\newcommand{\GVe}{\varepsilon}

\newcommand{\GGl}{\lambda}

\newcommand{\GGs}{\sigma}

\newcommand{\GGt}{\theta}

\newcommand{\GGw}{\omega}

\newcommand{\lA}{\left\{}
\newcommand{\rA}{\right\}}

\newcommand{\Gal}{\mathrm{Gal}}

\newcommand{\NN}{\mathrm{N}}

\newcommand{\red}{\mathrm{red}}

\usepackage{makeidx}
\title{On the $\mu$-invariant of Katz $p$-adic $L$ functions attached to imaginary quadratic fields and applications}
\author{{Hassan OUKHABA} and {St\'ephane VIGUI\'E}}

\makeindex

\newtheorem{cro}{cro}[section]

\newtheorem{lem}{lem}[section]
\newtheorem{pro}{pro}[section]

\newtheorem{teh}{teh}[section]

\newtheorem{cor}[cro]{Corollary}

\newtheorem{lm}[lem]{Lemma}
\newtheorem{pr}[pro]{Proposition}

\newtheorem{theo}[teh]{Theorem}

\begin{document}
\maketitle

\begin{abstract}
We extend to $p=2$ and $p=3$ the result of Gillard and Schneps which says that the $\mu$-invariant of Katz $p$-adic $L$ functions attached to imaginary quadratic fields is zero for $p>3$. The arithmetic interpretation of this fact is given in the introduction.
\end{abstract}

\noindent{\small\textbf{Mathematics Subject Classification (2010): 11G16, 11F85, 11S80, 11R23}} 
\\

\noindent{\small\textbf{Key words: Elliptic units, Coleman power series, $p$-adic $L$ functions, $\mu$-invariant, Iwasawa theory}}

\section{Introduction}

Let $k\subset\mathbb{C}$ be an imaginary quadratic field and let us denote by $\MCO_k$ its ring of integers. let $p$ be a prime number which splits completely in $k$, 
that is $p\MCO_k=\MFp\bar{\MFp}$, where $\MFp$ and $\bar{\MFp}$ are the prime ideals of $\MCO_k$ above $p$. 
Let us also fix $K$ a finite abelian extension of $k$. 
Then we denote by $K_\infty$ (resp. $k_\infty$)  the unique $\MBZ_p$-extension of $K$ (resp. $k$) unramified outside of $\MFp$. Of course we have $K_\infty=Kk_\infty$.
We are interested in the Galois group  
\begin{equation*}
X_\infty:=\mr{Gal}(M_\infty/K_\infty),
\end{equation*}
where $M_\infty$ is the maximal abelian $p$-extension of $K_\infty$ unramified outside of $\MFp$. 
Let us write $\Gamma$ for the Galois group of $K_\infty/K$, $\Gamma:=\mr{Gal}(K_\infty/K)$, 
then $X_\infty$ is naturally a module over the Iwasawa algebra $\Lambda:=\MBZ_p[[\Gamma]]$. 
Moreover, Greenberg has proved in \cite[Remark at the end of \S 4]{Gre78} that $X_\infty$ is a finitely generated torsion $\Lambda$-module 
and has no nontrivial finite $\Lambda$-submodule. 
Using an isomorphism $\Lambda\simeq\MBZ_p[[T]]$ we may write a generator of the characteristic ideal of $X_\infty$ as $p^\mu P(T)$, 
where $\mu$ is a non-negative integer and $P(T)$ is a distinguished polynomial. 

\begin{theo}\label{gillardshneps}$\left(Gillard\ for\ all\ k, \cite{gillard85,gillard87}\ and\ Schneps\ for\ k\ principal, \cite{schneps87}\right)$ The invariant $\mu$ vanishes for $p\geq 5$.
\end{theo} 
\noindent The ingredients used by Gillard to prove this important result are essentially:\par 
1) The algebraic independance of formal multiplications on elliptic curves.\par 
2) The $p$-adic $L$ functions for $p\geq5$.\\
The first ingredient is absolutely general and does not use the hypothesis $p\geq5$. It is the analogue of the algebraic independance result used by Sinnott in \cite{Sinnott84} to obtain a new proof of the theorem of Ferrero and Washington, cf. \cite{FerreroWashington79}. Moreover, the $p$-adic $L$ functions used by Gillard were constructed and studied by him only for $p\geq5$. Fortunately these functions are available for any $p$. They are constructed for instance in \cite{deshalit87}. In this paper we prove
\begin{theo}\label{theomu} $\mu=0$ for all $p$.
\end{theo}
Our proof uses the approach of Gillard. It is valid for any prime $p$. The informed reader knows the crutial role of the $p$-adic $L$-functions for our purpose. For example each of the Iwasawa invariants of $X_\infty$ is related to the corresponding invariants of the $p$-adic $L_{p,\MFf_{\chi_0}}$-functions defined in section \ref{fonctionLP}. This fact is proved in \cite[Theorem 2.1 Chapter III, page 109]{deshalit87} when $p\geq3$ and follows from (\ref{final}) in the general case. For these reasons we have devoted most of this text to the construction and study of the $p$-adic $L$ functions. In this we rely on the book of de Shalit, cf.\cite{deshalit87}. In particular we use a special class of CM elliptic curves, whose description is given in section \ref{Robert} below. The Robert elliptic $\psi$-functions obtained from theses curves give us both elliptic units and the associated Coleman power series. We transform these power series into integral measures and then we obtain our $p$-adic $L$-functions by integrating characters.  Let us observe that the elliptic curves used have complex multiplications by $\mathcal{O}_k$ and have good ordinary reduction at $\MFp$. This requires us to use {\bf generalized Weierstrass equations} to obtain integral models at any pre-determined prime $p$, especially at $p=2$ or $p=3$. Also we should mention that it is the first time Robert's $\psi$-functions are used for the construction of $p$-adic $L$ functions. De Shalit and Gillard have used their $12$-th power. See section \ref{Robert}.\par
Recently Ashay Burungale and Ming-lun Hsieh proved in \cite{BuHs13} that the $\mu$-invariant of $p$-adic Heck $L$-functions for CM fields vanishes for any ordinary prime $p>2$. On may ask if it is also the case for $p=2$.

\section{Notations and conventions}\label{notation}
Throughout this paper we denote by $\overline{\mathbb{Q}}$ the algebraic closure of $\mathbb{Q}$ in $\mathbb{C}$ and we see all our number fields as subfields of $\overline{\mathbb{Q}}$. By definition $k^{\mathit{ab}}$ is the abelain closure of $k$ in $\overline{\mathbb{Q}}$. We choose, once for all, an embedding 
\begin{equation*}
i_{\MFp}:\overline{\mathbb{Q}}\longrightarrow\mathbb{C}_p,
\end{equation*} 
such that $\mathfrak{p}$ is the prime ideal of $k$ defined by $i_{\MFp}$. We denote by $\MFP$ the prime of $\overline{\mathbb{Q}}$ determined by $i_{\MFp}$ and for a subfield $H$ of $\overline{\mathbb{Q}}$ we denote by $H_\MFP$ the completion of $i_{\MFp}(H)$ in $\MBC_p$. By $\log:\mathbb{C}_p^\times\longrightarrow\mathbb{C}_p$ we shall mean the $p$-adic logarithm map satisfying $\log(p)=0$.\par
Let $K$ be a finite abelian extension of $k$ and let $E$ be a subgroup of finite index in the group of units $\mathcal{O}_K^\times$ of $K$ such that $\varepsilon_1,\ldots,\varepsilon_{d-1}$ is a basis of $E$ modulo torsion ($d=[K:k]$). Let $\tau_1,\ldots,\tau_d$ be the embeddings of $K$ in $\mathbb{C}_p$ that induce $\mathfrak{p}$ on $k$. Let us remark that for every $i$ there exists a unique $\sigma_i\in\mathrm{Gal}(K/k)$ such that $\tau_i=i_{\mathfrak{p}}\circ\sigma_i$. Then we define the $\MFp$-adic regulator of $E$ by
\begin{equation*} 
 R_{\MFp}(E):=\det(\log \tau_i(\varepsilon_j))_{1\leq i,j\leq d-1}.
\end{equation*}
If $E=\mathcal{O}_K^\times$ then we write  $R_{\MFp}(K)$ instead of $R_{\MFp}(\mathcal{O}_K^\times)$.\par
Let us denote $h(K)$ the ideal class number of $K$, $\mu(K)$ the group of roots of unity in $K$ and $w_K$ the order of $\mu(K)$. If $\MFa$ is a fractional ideal of $k$ prime to the conductor of $K/k$ then we denote $(\MFa, K/k)$ the automorphism of $K/k$ associated to $\MFa$ by the Artin map. If $\mathfrak{f}$ is an integral ideal of $k$ then we write $k(\mathfrak{f})$ for the ray class field modulo $\mathfrak{f}$. Hence, $k(1)$ is the Hilbert class field of $k$. As always the number of roots of unity congruent to 1 modulo $\MFf$ will be denoted $w_{\MFf}$.

\section{Robert elliptic functions and Shimura curves}\label{Robert}
Let $L\subset L'$ be two lattices of $\MBC$ such that the index $[L':L]$ is prime to $6$, and let $\wp_L$ be the Weierstrass function attached to $L$.
G.\,Robert introduced in \cite{robert90} the elliptic function
\begin{equation}
\label{definitionpsi}
\psi\left( z; L, L' \right) := \GGd\left( L, L' \right) 
\prod_{\rho\in \MCZ\left( L,L' \right)} \left( \wp_L\left( z \right) - \wp_L\left( \rho \right) \right)^{-1},
\end{equation}
where $\GGd\left( L, L' \right)$ is the canonical $12$-th root of $\GGD\left( L \right)^{ \left[L':L\right] }/\GGD\left( L' \right)$ defined in \cite{robert90a},
and $\MCZ\left(L, L' \right)$ is any complete representative system of $\left(L'/L\right)\setminus\{0\}$ 
modulo $\{-1,1\}$. The period lattice of $\psi\left( z; L, L' \right)$ is equal to $L$ and its divisor is the sum
\begin{equation}\label{diviseurpsi}
 ([L':L]-1)(0)_L-\sum_{\rho\in\MCZ\left( L,L' \right)}\pm\rho.
\end{equation}
For any $\GGl\in\MBC^\times $, we have
\begin{equation}
\label{psihomogene}
\psi\left( \GGl z ; \GGl L , \GGl L' \right) = \psi\left( z;L,L' \right).
\end{equation} 
Now assume that $L$ and $L'$ are $\MCO_k$-submodules of $\MBC$.
By \cite[(1) p.\,232]{robert90}, for any ideal $\MFa$ of $\MCO_k$ prime to $6$ satisfying $L'\cap\MFa^{-1}L=L$, we have
\begin{equation}
\label{distributionpsi}
\psi\left( z ; L' , \MFa^{-1}L' \right) = \prod_{\rho} \psi\left( z + \rho ; L , \MFa^{-1}L \right),
\end{equation} 
where the product is over any complete system of representatives system of $L'/L$. Moreover, if $\rho\neq0$ is an $\MFm$-torsion point of $\MBC/L$ for some ideal $\MFm$ of $\MCO_k$ prime to $\mathfrak{a}$, then $\psi\left( \rho;L,\MFa^{-1}L \right)\in k(\MFm)^{\times}$ and for all ideal $\MFb$ prime to $\MFm$ we have
\begin{equation}
\label{psigalois}
\psi\left( \rho;L,\MFa^{-1}L \right)^{\left( \MFb,k(\MFm)/k \right)} 
= \psi\left( \rho; \MFb^{-1}L, \left(\MFa\MFb\right)^{-1}L \right).
\end{equation} 
In particular if $\MFm=\MFn\MFq$ where $\MFn$ is an ideal of $\mathcal{O}_k$ and $\MFq$ is a prime ideal of $\MCO_k$, then we have
\begin{equation}
\label{normpsi}
\NN_{k(\MFn\MFq)/k(\MFn)}\left(\psi\left( \rho;L,\MFa^{-1}L \right)\right)^{\frac{w_{\MFn}}{w_{\MFn\MFq}}}
= 
\begin{cases}
\psi\left( \rho;\MFq^{-1}L,\MFa^{-1}\MFq^{-1}L \right) & if\ \MFq\mid\MFn\\
\psi\left( \rho;\MFq^{-1}L,\MFa^{-1}\MFq^{-1}L \right)^{1-(\MFq,k(\MFn)/k)^{-1}} & if\ \MFq\nmid\MFn.
\end{cases}
\end{equation} 
Theses special values of Robert $\psi$-functions are used for instance in \cite{oukhaba2006, oukhaba2007} to construct groups of elliptic units with very ineresting properties. We also draw the attention of the reader that theses functions have an analogue in positive caracteristique and are closely related to Drinfel'd modules. See \cite{oukhaba1991} for more details.\par
\bigskip
To construct $p$-adic $L$ functions it is crutial to use lattices $L$ associated to Shimura elliptic curves, that is those curves satisfying the properties $E1$ and $E3$ given below. Therefore, from now on we fix an integral ideal $\MFf$ of $\mathcal{O}_k$ prime to $\MFp$ such that $w_{\MFf}=1$, and we set
\[  F:=k(\mathfrak{f}).\]
Also we fix an elliptic curve $E$ such that
\begin{equation}\label{shimura}
\begin{split}
 &E1.\ E\ \mathit{has\ complex\ multiplication\ by}\ \mathcal{O}_k.\\
 &E2.\ E\ \mathit{is\ defined\ over}\ k(\mathfrak{f}).\\
 &E3.\ \mathit{All\ the\ points\ of\ finite\ order\ of}\ E\ \mathit{are\ rational\ over}\ k^{\mathit{ab}}.\\
 &E4.\ E\ \mathit{has\ good\ reduction\ at\ all\ the\ primes\ of}\ k(\mathfrak{f})\ \mathit{not\ dividing}\ \MFf\mathcal{O}_F.\\
\end{split}
\end{equation}
The existence of elliptic curves defined over $k^{ab}$ and satisfying $E1$ and $E3$ is proved in \cite[page 216]{Shimura71}. The existence of elliptic curves $E$ with the additional conditions $E2$ and $E4$ is proved for instance in \cite[Chap. II, \S 1.4 Lemma]{deshalit87}. Let $R$ be the integral closure in $F$ of the discrete valuation ring $S^{-1}\mathcal{O}_k$, where  $S:=\mathcal{O}_k\setminus\mathfrak{p}$. We fix a minimal generalized Weierstrass model of $E$ with coefficients in $R$, 
\begin{equation}
\label{equacourbe}
y^2 + a_1 xy + a_3 y = x^3 + a_2 x^2 + a_4 x + a_6.
\end{equation}
By minimal we mean that the discriminant of such an equation is a unit at all the prime ideals of $F$ above $\mathfrak{p}$. Let $\omega_E$ be the usual holomorphic invariant differential attached to (\ref{equacourbe}), that is
\[ \omega_E=\frac{dx}{2y+a_1x+a_3}.\]
The pair $(E,\omega_E)$ determines a unique $\MCO_k$-lattice $L$ of $\MBC$ such that the following isomorphism holds,
\[\GGt_{\infty,E} : \MBC/L \xrightarrow{\quad\sim\quad} E\left( \MBC \right),\]
where for all $z\in\MBC\setminus L$, $\GGt_{\infty,E}(z)$ is the unique point whose coordinates $x_\infty (z)$ and $y_\infty (z)$ are given by 
\begin{equation}\label{xyetstrass}
x_\infty (z):=\wp_L\left( z \right)-b_2/12 \quad\text{and}\quad y_\infty (z):=\left( \wp_L'\left( z \right) - a_1 x_\infty (z) - a_3\right)/2,
\end{equation}
where $b_2:=a_1^2+4a_2$. The coefficients of the differential equation $(\wp_L')^2=4\wp_L^3-g_2(L)\wp_L-g_3(L)$ satisfied by $\wp_L$ are given by
\begin{equation}\label{g2g3}
g_2(L)=12\Bigl(\frac{b_2}{12}\Bigr)^2-2b_4\quad\mathrm{and}\quad g_3(L)=\frac{b_2b_4}{6}-b_6-8\Bigl(\frac{b_2}{12}\Bigr)^3,
\end{equation}
where $b_4:=2a_4+a_1a_3$ and $b_6:=a_3^2+4a_6$. Moreover, the discriminant of the equation (\ref{equacourbe}) is
\begin{equation}\label{discriminant}
\Delta=-b_2^2b_8-8b_4^3-27b_6^2+9b_2b_4b_6=g_2^3-27g_3^2=\Delta(L),
\end{equation}
where $b_8:=a_1^2a_6+4a_2a_6-a_1a_3a_4+a_2a_3^2-a_4^2$.\par 
Let us fix an isomorphism $\iota:k\longrightarrow\mathrm{End}_{\mathbb{Q}}(E):=\mathbb{Q}\otimes_{\mathbb{Z}}\mathrm{End}(E)$ such that
\begin{equation}\label{polarisation}
\iota(\mu)^*\omega_E=\mu\omega_E,\quad\mathrm{for\ all}\ \mu\in k,
\end{equation}
If $\sigma\in\mathrm{Aut}(\mathbb{C})$ then the conjugate $E^{\sigma}$ of $E$ has also complex multiplications by $\mathcal{O}_k$. Thus $E$ is isogenous to $E^{\sigma}$ by \cite[Proposition 4.9]{Shimura71}. Now suppose that $\sigma\in\mathrm{Gal}(\mathbb{C}/k)$ and that the restriction of $\sigma$ to $F$ is the automorphism $\sigma_{\mathfrak{a}}=(\mathfrak{a},F/k)$ of $F/k$ associated by the Artin map to some integral ideal of $k$ prime to $\mathfrak{f}$, then by Shimura reciprocity law there exists an isogeny $\lambda_{\mathfrak{a}}:E\longrightarrow E^{\sigma_{\mathfrak{a}}}$ uniquely determined by the condition 
\begin{equation}\label{caracterisation}
 \rho^{\sigma}=\lambda_{\mathfrak{a}}(\rho),
\end{equation}
for any torsion point of $E$ of order prime to $\mathfrak{a}$, \cite[Theorem 5.4]{Shimura71}. It is easy to see from (\ref{caracterisation}) that $\lambda_{\mathfrak{a}}^\tau=\lambda_{\mathfrak{a}}$, for any $\tau\in\mathrm{Gal}(\mathbb{C}/F)$. This proves that $\lambda_{\mathfrak{a}}$ is defined over $F$. Let $\omega_E^{\sigma_{\mathfrak{a}}}$ be the image of  $\omega_E$ by $\sigma_{\mathfrak{a}}$, then there exists $\Lambda(\mathfrak{a},L)\in F$ such that
\begin{equation}\label{facteur}
 \lambda_{\mathfrak{a}}^*\omega_E^{\sigma_{\mathfrak{a}}}=\Lambda(\mathfrak{a},L)\omega_E.
\end{equation}
Indeed, $\lambda_{\mathfrak{a}}^*\omega_E^{\sigma_{\mathfrak{a}}}$ is a holomorphic differential on $E$ defined over $F$, and we know that such differentials form an $F$-vector space of dimension $1$. Also we deduce from \cite[Theorem 5.4]{Shimura71} that the lattice attached to $(E^{\sigma_{\mathfrak{a}}},\omega_E^{\sigma_{\mathfrak{a}}})$ is $L_{\mathfrak{a}}:=\Lambda(\mathfrak{a},L)\mathfrak{a}^{-1}L$ and the following diagram commutes

\begin{equation*}
\begin{CD}
\mathbb{C}@>>>\mathbb{C}/L@>\theta_{\infty,E}>> E\\
@V\Lambda(\mathfrak{a},L)VV @VVV @VV\lambda_{\mathfrak{a}}V\\
\mathbb{C}@>>>\mathbb{C}/L_{\mathfrak{a}}@>\theta_{\infty,E^{\sigma_{\mathfrak{a}}}}>> E^{\sigma_{\mathfrak{a}}}.
 \end{CD}
\end{equation*}
In particular,
\begin{equation}\label{g2Lag3La}
g_n(L)^{\sigma_{\mathfrak{a}}}=g_n(L_{\mathfrak{a}})=\Lambda(\mathfrak{a}, L)^{-2n}g_n(\mathfrak{a}^{-1}L),\quad n\in\{2,3\}.
\end{equation}
Also we see by using (\ref{caracterisation}) and the above commutative diagram that $\Lambda(x\mathcal{O}_k, L)=x$ for any $x\in\mathcal{O}_k$ satisfying $x\equiv1$ modulo $\MFf$. Another consequence of (\ref{caracterisation}) is the relation $\lambda_{\mathfrak{ab}}=\lambda_{\mathfrak{a}}^{\sigma_{\MFb}}\circ\lambda_{\mathfrak{b}}=\lambda_{\mathfrak{b}}^{\sigma_{\MFa}}\circ\lambda_{\mathfrak{a}},$ from which we derive the cocycle rule
\begin{equation*}
\Lambda(\mathfrak{ab},L)=\Lambda(\mathfrak{a},L)^{\sigma_{\MFb}}\Lambda(\mathfrak{b},L),
\end{equation*}
for all the integral ideals $\MFa$ and $\MFb$ prime to $\MFf$. This allows us to define $\Lambda(\mathfrak{a},L)$ even for fractional ideals of $k$ that are prime to $\MFf$. Let us remark that $E^{\sigma_{\MFb}}$ also satisfies (\ref{shimura}) and
\begin{equation}\label{lambdala}
\Lambda(\MFa, L_{\MFb})=\Lambda(\mathfrak{a},L)^{\sigma_{\MFb}}.
\end{equation}

Let $\widehat{E}$ be the one parameter formal group of $E$ at the origine, with parameter $t:=-x/y$. $\widehat{E}$ is defined over $\MBZ[a_1,\ldots,a_6]$ and we have
\[\widehat{E}(t(P),t(Q))=t(P+Q).\]
To go further we consider $E$ as defined on $\mathbb{C}_p$ by using the embeding $i_{\MFp}$. Let $M\subset\MBC_p$ be any complete field containing $F_{\MFP}$. Let $\MCO_{M}$ be the valuation ring of $M$ and let $\MFp_{M}$ be the maximal ideal of $\MCO_{M}$. We write $\widehat{E}(\MFp_M)$ for $\MFp_{M}$ endowed with the group law given by $\widehat{E}$.
Let $E_1\left(M\right)$ be the kernel of $E\left(M\right)\xrightarrow{}\tilde{E}\left(\mathbb{M}_{\MFP}\right)$,
where $\mathbb{M}_{\MFP}:=\MCO_{M}/\MFp_{M}$ and $\tilde{E}$ is the reduction of $E$.
When $\MCO_{M}$ is a discrete valuation ring then it is proved, for instance in \cite[Chap. VII, Proposition 2.2.]{Silverman92}, that there is a group isomorphism
\begin{eqnarray*}
\GGt_{\MFp,E}^{M} :  \widehat{E}(\MFp_{M})& \longrightarrow & E_1\left(M \right)\\
                     t & \longmapsto & (t/s(t), -1/s(t))\\
\end{eqnarray*}
where $s(t)=t^3+i_{\MFp}(a_1)t^4+i_{\MFp}\left( a_2+a_1^2 \right)t^5+...$ is the unique series with coefficients in $\MCO_{F_{\MFP}}$ such that $s(0)=0$ and
\[s(t)=t^3+ i_{\MFp}(a_1) ts(t) + i_{\MFp}(a_2) t^2s(t) + i_{\MFp}(a_3) t(z)^2 + i_{\MFp}(a_4) ts(t)^2 + i_{\MFp}(a_6) s(t)^3 \qquad\text{for all $t\in \MFp_{M}$.}\]
In other words, we have the power series expansions
\begin{equation}\label{aetb}
i_{\MFp}(x_\infty)=t^{-2}a(t), \quad\quad i_{\MFp}(y_\infty)=-t^{-3}a(t),\quad\quad a(t):=t^3/s(t).
\end{equation}
We remark that
\begin{equation}\label{coeffdea}
a(t)\in\MCO_{F_{\MFP}}[[t]]\ \mathrm{and}\ a(t)=1-i_{\MFp}(a_1)t-i_{\MFp}(a_2)t^2\ \mathrm{modulo}\ t^3\MCO_{F_{\MFP}}[[t]].
\end{equation}
Actually one may also define a map $\GGt_{\MFp,E}^{\MBC_p}$ in a similar way as above and prove that it is an isomorphism by using essentially the same arguments used in the case of local fields. We shall write $\GGt_{\MFp,E}$ instead of $\GGt_{\MFp,E}^{\MBC_p}$. Let $\overline{\MBQ}\left(\wp_L,\wp_L'\right)$ be the subfield of the field of $L$-elliptic functions generated by $\overline{\MBQ}$, $\wp_L$ and $\wp_L'$. Let
\[\mathit{C_E}:\overline{\MBQ}\left(\wp_L,\wp_L'\right)\xrightarrow{\qquad}\MBC_p((t)),\]
be the unique homomorphism of rings extending the map $i_{\MFp}$, such that $C_E(x_\infty)=t^{-2}a(t)$ and $C_E(y_\infty)=-t^{-3}a(t)$. In the sequel we will also use the notation $\hat{f}$ for $C_E(f)$. The normalized invariant differential of $\widehat{E}/\mathcal{O}_{F_{\MFP}}$ is given by
\[\hat{\omega}_E=\frac{d\hat{x}_\infty}{2\hat{y}_\infty+i_{\MFp}(a_1)\hat{x}_\infty+i_{\MFp}(a_3)}=\lambda'(t)dt,\]
where $\lambda\in t+t^2F_{\MFP}[[t]]$ is the unique logarithm map of $\widehat{E}$ satisfying $\lambda'(0)=1$, cf. \cite[Chap.IV, \S4 and \S5]{Silverman92}. It is also well known that $\lambda'(t)\in\MCO_{F_{\MFP}}[[t]]^{\times}$ as proved in \cite[\S 5.8]{haz12}. It is interesting to remark that the equality $t=-x/y$ and (\ref{xyetstrass}) allow us to expand $t$ as a power series in $z$, $t=\varphi(z)$. The logarithm $\lambda$ is the inverse of $\varphi$, that is $\varphi\circ\lambda(t)=t$. Thus, if $P(f)$ is the Laurent development of $f$ at $z=0$ then
\begin{equation}\label{lambdaetchapeau}
\hat{f}(t)=(P(f)\circ\lambda)(t).
 \end{equation}

Let us recall one more helpful relation, used in the proof of Propositions \ref{psicoleman1} and \ref{psicoleman}. If $f\in \overline{\MBQ}\left(\wp_L,\wp_L'\right)$ and $P\in E_1\left(\overline{\MBQ}\right)$, 
where $E_1\left(\overline{\MBQ}\right):=E_1\left(\MBC_p\right)\cap E\left(i_{\MFp}(\overline{\MBQ})\right)$, then we have
\begin{equation}
\label{vivelesabeilles}
 f\left( \GGt_{\infty,E}^{-1}\left( P \right) \right) 
= i_\MFp^{-1}\left( \hat{f}\left( \GGt_{\MFp,E}^{-1}\left( P \right) \right) \right),
\end{equation} 
whenever it is defined.\par
If $\MFa$ is an integral ideal of $\mathcal{O}_k$ then we denote by $E[\MFa]$ the group of points of $E$ annihilated by all $a\in\MFa$. By definition $F(E[\MFa])$ is the abelian extension of $k$ generated by $F$ and by the coordinates of the points of $E$ in $E[\MFa]$. It is proved in many references that $F(E[\MFa])=k(\MFa)$ for any elliptic curve $E$ satisfying $E1, E2$ and $E3$ in (\ref{shimura}) and for any $\MFa$ such that $\MFf\vert\MFa$. See \cite{arthaud78,CW77,deshalit87, goldschapp81} for the details. We point out that in \cite{arthaud78, CW77} the imaginary quadratic field $k$ is assumed to be principal. We deduce from above that if $\MFa$ is prime to $\MFf$ then $F(E[\MFa])=k(\MFa\MFf)$.
\begin{pr}\label{PsiOL1}
Let $\MFm\neq \MCO_k$ be an integral ideal prime to $\MFp$, 
and let $\Omega\in \MBC$ be a primitive $\MFm$-torsion point of $\MBC/L$ $\left(\MFm=\Omega^{-1}L\cap\mathcal{O}_k\right)$.
For any ideal $\MFb$ prime to $6\MFm\MFf\MFp$, the $L$-elliptic function $\Psi_{\Omega,\MFb}^L$ defined by
$\Psi_{\Omega,\MFb}^L(z):=\psi\left( z+\Omega; L,\MFb^{-1}L \right)$ is such that 
\begin{equation}\label{coefficients}
 \hat{\Psi}_{\Omega,\MFb}^L(t)\in \MCO_{H_{\MFP}} \left[ \left[t\right] \right]^\times, 
\end{equation}
where $H=F(E[\MFm])$.
\end{pr}
\begin{proof} It is well known that $\hat{\Psi}_{\Omega,\MFb}^L(t)^{12}\in\MCO_{H_{\MFP}} \left[ \left[t\right] \right]^\times$ for $L=\Omega\MFf$, $\MFm=\MFf$ and $\Omega=\Omega$. This is proved in \cite[chap II, \S 4.9 Proposition (i), page 72]{deshalit87}. We point out that $\hat{\Psi}_{\Omega,\MFb}^L(t)^{12}\in \MCO_{H_{\MFP}} \left[ \left[t\right] \right]^\times$ is also proved in earlier works, but under some restrictive hypotheses, cf. \cite{CW77, arthaud78, rubin81, yager82}. To check (\ref{coefficients}) we observe first that the function
\begin{equation*}
\Gamma(z):=\delta(L,\MFb^{-1}L)/\Psi(z,L,\MFb^{-1}L)
\end{equation*}
is a polynomial in $x_\infty(z)$ with coefficients in $R$ and leading coefficient  equal to 1. Indeed, we have $\wp_L\left( z \right) - \wp_L\left(z' \right)=x_\infty(z)-x_\infty(z')$, for any $z,z'\in\mathbb{C}\setminus L$. Moreover, the $x$-coordinates $x_\infty\left(\gamma\right), \gamma\in\MCZ:=\MCZ(L,\MFb^{-1}L)$, are permuted by $\mathrm{Gal}(\mathbb{C}/F)$. Since they are integral at all the primes above $\MFp$ our claim is proved.
Now if we replace $x_\infty(z)$ by $x_\infty(z+\Omega)$ the addition law on $E$ gives
\begin{equation*}
x_\infty(z+\Omega)= 
\left( \frac{ y_\infty (z)-y_\infty (\Omega)}{x_\infty (z) -x_\infty (\Omega)}\right)^2 + a_1\frac{ y_\infty (z)-y_\infty (\Omega)}{x_\infty (z) -x_\infty (\Omega)} 
-a_2-x_\infty (z) -x_\infty (\Omega) . 
\end{equation*}
Therefore we deduce that $\Gamma(z+\Omega)$ is a rational function of $x_\infty (z)$ and $y_\infty (z)$ with $\MFP$-integral coefficients in $H$.
Let $f(z):=x_\infty(z+\Omega)-x_\infty(\gamma)$, for a fixed $\gamma\in\MCZ$. Then, by applying (\ref{aetb}) and (\ref{coeffdea}) we see that  $\hat{f}(t)=i_{\MFp}(x_\infty(\Omega)-x_\infty(\gamma)) + t\MCO_{H_{\MFP}} \left[ \left[t\right] \right]$. But $i_{\MFp}(x_\infty(\Omega)-x_\infty(\gamma))$ is a $\MFP$-unit, for otherwise, the images $\widetilde{P}$ and $\widetilde{Q}$ in $E(H_{\MFP})$ of the points $P:=(x_\infty(\Omega),y_\infty(\Omega))$ and $Q:=(x_\infty(\gamma),y_\infty(\gamma))$ would be such that $\widetilde{P}+\widetilde{Q}\in E_1(H_{\MFP})$ or $\widetilde{P}-\widetilde{Q}\in E_1(H_{\MFP})$. But the point $\widetilde{P}+\widetilde{Q}\in E_1(H_{\MFP})$ means that $\widetilde{P}+\widetilde{Q}=O$ or its coordinates are not $\MFP$-integral. This is impossible as one may check by using \cite[Chapter VII, Theorem 3.4]{Silverman92}(we recall that $\Omega\in\MFm^{-1}L$ and $\gamma\in\MFb^{-1}L$).\par
Now that we have proved that $\hat{\Gamma}(z+\Omega)$ belongs to $\MCO_{H_{\MFP}} \left[ \left[t\right] \right]^\times$ we have to consider $\delta(L,\MFb^{-1}L)$. By \cite[Th\'eor\`eme 4]{robert90a} we know that
\begin{equation*}
\delta(L,\MFb^{-1}L)\in\MBQ(g_n(L),g_n(\MFb^{-1}L)))_{n\in\{2,3\}}.
\end{equation*}
But one sees immediately from (\ref{g2g3}) that $g_n(L)\in F$ for $n\in\{2,3\}$, and from (\ref{g2Lag3La}) that $g_n(\MFb^{-1}L)\in F$ for $n\in\{2,3\}$. The $12$-th power of $\delta(L,\MFb^{-1}L)$ is equal to
\[\frac{\GGD\left( L \right)^{N\MFb}}{\GGD\left(\MFb^{-1}L \right)}
=\GGD\left( L \right)^{N\MFb-1}\frac{\GGD\left( L \right)}{\GGD\left(\MFb^{-1}L \right)}.\]
By (\ref{discriminant}) and by our choice of $E$ the discriminant $\Delta(L)$ is a unit at all the primes of $F$ above $\MFp$. Moreover, it is well known that $\GGD\left( L \right)/\GGD\left(\MFb^{-1}L \right)$ belongs to $k(1)$, the Hilbert class field of $k$, and generates the ideal $\MFb\mathcal{O}_{k(1)}$. Since $\MFb$ is prime to $\MFp$ we deduce that $\delta(L,\MFb^{-1}L)$ is a unit at all the primes of $F$ above $\MFp$ too. This completes the proof of the proposition. 
\hfill $\square$ 
\end{proof}


\begin{cor}\label{PsiOL}
Let $\MFm\neq \MCO_k$ be an integral ideal of $\MCO_k$ prime to $\MFp$, and let $\Omega\in \MBC$ be a primitive $\MFm$-torsion point of $\MBC/L$.
For any ideals $\MFa$ and $\MFb$, such that $\MFa$ is prime to $6\MFm\MFf\MFp$ and $\MFb$ is prime to $6\MFa\MFm\MFf\MFp$, the function $\Psi_{\Omega,\MFa,\MFb}^L(z):=\psi\left( z+\Omega; \MFa^{-1}L,\MFa^{-1}\MFb^{-1}L \right)$ satisfies
\[\hat{\Psi}_{\Omega,\MFa,\MFb}^L(t)\in \MCO_{H_{\MFP}} \left[ \left[t\right] \right]^\times .\]
\end{cor}


\begin{proof} This is a direct consequence of Proposition \ref{PsiOL1} since we have the identity
\begin{equation}\label{multiplicativity}
\Psi_{\Omega,\MFa\MFb}^L(z)=\Psi_{\Omega,\MFa}^L(z)^{N(\MFb)}\Psi_{\Omega,\MFa,\MFb}^L(z),
\end{equation}
which is a particular case of the multiplicativity formula satisfied by Robert Functions, \cite[Corollaire 3]{robert90a}.
\hfill $\square$ 
\end{proof}
\section{Coleman series of elliptic units}\label{colemanseries}
Let $k^{ur}_{\MFp}$ be the maximal unramified extension of $k_{\MFp}:=k_{\MFP}$ in $\mathbb{C}_p$. As it is well known the group $\mathrm{Gal}(k^{ur}_{\MFp}/k_{\MFp})$ is topologically generated by the Frobenius automorphism $\Phi$, uniquely determined by the condition
\[\Phi(\xi)=\xi^p,  \]
for any root of unity $\xi$ of order prime to $p$. $\Phi$ is also caracterized by the fact that, for any integral element $x$ in $k^{ur}_{\MFp}$ the image $\Phi(x)$ is congruent to $x^p$ modulo the maximal ideal of $k^{ur}_{\MFp}$. Let us remark that the restriction of $\Phi$ to $F$ (actually to $i_{\MFp}(F)$) is just the Frobenius automorphism $\sigma_{\MFp}:=(\MFp, F/k)$ of $F/k$ at $\MFp$. As we have seen in the previous section, there is an isogeny $\lambda_{\MFp}: E\longrightarrow E^{\sigma_{\MFp}}$ from which we derive a morphism of formal groups
\[f_E: \widehat{E}\longrightarrow \widehat{E^{\sigma_{\MFp}}}.\]
In other words, $f_E\in\mathcal{O}_{F_{\MFP}}[[t]]$ and satisfies $f_E\circ\widehat{E}=\widehat{E^{\sigma_{\MFp}}}\circ f_E$. Further, the equality $\widehat{E^{\sigma_{\MFp}}}=\widehat{E}^{\Phi}$ and the fact that
\[f_E(t)\equiv\Lambda(\MFp,L)t\ \mathrm{modulo}\ t^2\mathcal{O}_{F_{\MFP}}[[t]]\quad\mathrm{and}\quad f_E(t)\equiv t^p\ \mathrm{modulo}\ \MFp_{F_{\MFP}}[[t]]\]
clearly shows that $\widehat{E}$ is a Lubin-Tate formal group relative to the extension $F_{\MFP}/k_{\MFp}$ as defined in \cite[Chapter I, \S 1.2 and 1.3]{deshalit87}. By \cite[Chapter I, Proposition 1.5]{deshalit87}, for any $a\in\mathcal{O}_{k_{\MFp}}$ there exists a unique power series $[a]_E\in\mathcal{O}_{F_{\MFP}}[[t]]$ such that 

\begin{eqnarray*}
 &(i)& [a]_E(t)\equiv at\ \mathrm{modulo}\ t^2\mathcal{O}_{F_{\MFP}}[[t]].\\
 &(ii)& [a]_E^{\Phi}\circ f_E=f_E\circ[a]_E.\\
 &(iii)& [a]_E\circ\widehat{E}=\widehat{E}\circ[a]_E.
\end{eqnarray*}
Moreover,
\begin{equation}\label{morphism}
\theta_{\MFp,E}([i_{\MFp}(a)]_E(z))=a.\theta_{\MFp,E}(z),\ \mathrm{for\ any}\ z\in\theta_{\MFp,E}^{-1}(E_1(\overline{\mathbb{Q}}))\ \mathrm{and\ any}\ a\in\MCO_k.
\end{equation}
To simplify notation we shall write $f$ instead of $f_E$ and $[a]$ for $[a]_E$. Let $W_f^n:=\{\alpha\in\MFp_{\mathbb{C}_p},\ [a](\alpha)=0,\ \mathrm{for\ all}\ a\in\MFp^n\}$. Then (\ref{morphism}) implies that
\[\theta_{\MFp,E}(W_f^n)=E[\MFp^n]\]
Hence if we let $F_n:=F(E[\MFp^n])$ then, as we have explained just before Proposition \ref{PsiOL1}, the field $F_n$ is equal to the ray class field $k(\MFf\MFp^n)$; and by the equality above  
\[ F_{\MFP}(W_f^n)=\widetilde{F}_n,\]
where $\widetilde{F}_n:=(F_n)_{\MFP}$ denote the completion of $i_{\MFp}(F_n)$ in $\mathbb{C}_p$. To go further we need to produce primitive $\MFp^n$-torsion points of the elliptic curve $E^{(n)}:=E^{\Phi^{-n}}$. We fix throughout this section
\begin{equation}\label{periodecomplexe}
\Omega\ \mathit{a\ primitive}\ \MFf-\mathit{torsion\ point\ of}\ \mathbb{C}/L.
\end{equation}
This assumption means that $L=\Omega\mathfrak{c}^{-1}\MFf$, where $\MFc$ is an ideal of $\MCO_k$ prime to $\MFf$. Let us also make the following hypothesis
\begin{equation}\label{periodecomplexe2}
The\ ideal\ \mathfrak{c}\ is\ prime\ to\ \MFp.
\end{equation}
Let $x\in\MFp$ be such that $x\equiv1$ modulo $\MFf$. Then $\Phi^{-n}=\sigma_{\MFa_n}=(\MFa_n,F/k)$, where $\MFa_n:=(x\MFp^{-1})^n$, and \[\rho_n:=\lambda_{\MFa_n}(\theta_{\infty,E}(\Omega))=\theta_{\infty,E^{(n)}}(\Lambda(\MFa_n,L)\Omega)\]
is a primitive $\MFf$-torsion point of $E^{(n)}=E^{\sigma_{\MFa_n}}$. Let
\[u_n:=\Lambda(\MFa_n,L)x^{-n}\Omega-\Lambda(\MFa_n,L)\Omega=\Lambda(\MFp^{-n},L)(1-x^n)\Omega.\] 
Then $\theta_{\infty,E^{(n)}}(u_n)$ is a primitive $\MFp^n$-torsion point of $E^{(n)}$ (the origine of $E$ if $n=0$). Let us remark that the lattice $L_{\mathfrak{a}_n}$ associated to $(E^{(n)},\omega_E^{\Phi^{-n}})$ is independant of $x$, and in fact we have
\[L_{\mathfrak{a}_n}=\Lambda(\MFp^{-n},L)\MFp^nL.\]
Moreover $u_n$ modulo $L_{\mathfrak{a}_n}$ is independant of $x$ too.\par
For any $n\in\mathbb{Z}$ and any $g\in\mathcal{O}_{k^{ur}_{\MFp}}[[t]]$ we denote $\Phi^{n}(g)$ or $g^{\Phi^n}$ the power series obtained by applying $\Phi^{n}$ to the coefficients of $g$. Also we set
\[\widehat{E}_{alg}:=\theta_{\MFp,E}^{-1}(E_1(\overline{\mathbb{Q}})). \]

\begin{lm}\label{baseTate} The elements $\pi_n:=\theta_{\MFp,E^{(n)}}^{-1}(\theta_{\infty,E^{(n)}}(u_n))$ ($\pi_0=0$) satisfy $\pi_n\in W_{\Phi^{-n}(f)}^n\backslash W_{\Phi^{-n}(f)}^{n-1}$ and $\Phi^{-n}(f)(\pi_n)=\pi_{n-1}$.
\end{lm}
\begin{proof} The isogeny $\lambda_{\MFp}^{\Phi^{-n}}:E^{(n)}\longrightarrow E^{(n-1)}$ sends $\theta_{\infty,E^{(n)}}(u_n)$ to $\theta_{\infty,E^{(n-1)}}(u_{n-1})$ and induce the following commutative diagramm
\begin{equation*}
\begin{CD}
(\widehat{E^{(n)}})_{alg}@>\theta_{\MFp,E^{(n)}}>>E^{(n)}_1(\overline{\mathbb{Q}}) \\
@V\Phi^{-n}(f)VV  @VV\lambda_{\MFp}^{\Phi^{-n}}V\\
(\widehat{E^{(n-1)}})_{alg}@>\theta_{\MFp,E^{(n-1)}}>>E^{(n-1)}_1(\overline{\mathbb{Q}}). 
\end{CD}
\end{equation*}
This implies the equality $\Phi^{-n}(f)(\pi_n)=\pi_{n-1}$.\hfill $\square$
\end{proof}
As proved in \cite[Chapter I, Theorem 2.2, page 13]{deshalit87}, for any norm-coherent sequence $\beta=(\beta_n)$ with $\beta_n$ in ($\mathcal{O}_{\widetilde{F}_n}^\times$), there exists a unique $F_{\Omega,\beta}\in\MCO_{F_\MFP}\left[\left[t\right]\right]^\times$ 
such that for any integer $n\geq1$,
\begin{equation}\label{Coleman}
\Phi^{-n}(F_{\Omega,\beta})\left(\pi_n \right) = \beta_n.
\end{equation} 
\begin{pr}
\label{psicoleman1} Let $\MFb$ be an ideal of $\MCO_k$ prime to $6\MFf\MFp$. Then the Coleman power series associated to the projective system of elliptic units $v:=(i_\MFp\left( \psi\left(\Omega;\MFp^{n}L,\MFb^{-1}\MFp^{n}L  \right) \right)_{n\geq1}$ is $F_{\Omega,v}=\hat{\Psi}^L_{\Omega,\MFb}$.
\end{pr}
\begin{proof} Let us first remark that for any $n\geq1$
\[\psi\left(\Omega;\MFp^{n}L,\MFb^{-1}\MFp^{n}L  \right)= \psi\left(1;\MFc^{-1}\MFp^{n}\MFf,\MFc^{-1}\MFb^{-1}\MFp^{n}\MFf  \right)=\psi\left(1;\MFp^{n}\MFf,\MFb^{-1}\MFp^{n}\MFf  \right)^{(\MFc, k(\MFp^n\MFf)/k)},\]
is a unit of $k(\MFp^n\MFf)$ as is well known. Moreover, the power series $\hat{\Psi}_{\Omega,\mathfrak{b}}^L$ lies in $\MCO_{F_\MFP}\left[\left[t\right]\right]^\times$, thanks to Proposition \ref{PsiOL1}. Further, if we choose $x$ such that $\MFa_n$ is prime to $\MFb$ then the definition of Robert $\psi$-functions and (\ref{caracterisation}) give the equality  $\Phi^{-n}(\hat{\Psi}_{\Omega,\mathfrak{b}}^L)=\hat{\Psi}_{\Omega_n,\mathfrak{b}}^{L_{\mathfrak{a}_n}}$, where $\Omega_n:=\Lambda(\MFa_n,L)\Omega$. Now an easy computation using (\ref{vivelesabeilles}) gives
\[\hat{\Psi}_{\Omega_n,\mathfrak{b}}^{L_{\mathfrak{a}_n}}(\pi_n)=i_\MFp(\psi(\Omega;\mathfrak{p}^nL,\mathfrak{b}^{-1}\mathfrak{p}^nL)).\]  
\hfill $\square$
\end{proof}
\begin{pr}
\label{psicoleman}
Let $\MFa$ be an ideal of $\MCO_k$, prime to $6\MFf\MFp$. Let $\MFb$ be an ideal of $\MCO_k$ prime to $6\MFa\MFf\MFp$.
Then the Coleman power series associated to the projective system of elliptic units $v:=(i_\MFp\left( \psi\left(\Omega;\MFa^{-1}\MFp^{n}L,\MFb^{-1}\MFa^{-1}\MFp^{n}L  \right) \right)_n$ is $F_{\Omega,v}=\hat{\Psi}^L_{\Omega,\MFa,\MFb}$
\end{pr}

\begin{proof} From the multiplicativity formula (\ref{multiplicativity}) and the fact that $\Phi^{-n}(\hat{\Psi}_{\Omega,\mathfrak{c}}^L)=\hat{\Psi}_{\Omega_n,\mathfrak{c}}^{L_{\mathfrak{a}_n}}$ for any integral ideal $\MFc$ prime to $6\MFp$ we deduce the relation
\[\Phi^{-n}(\hat{\Psi}_{\Omega,\mathfrak{a},{b}}^L)=\hat{\Psi}_{\Omega_n,\mathfrak{a},\MFb}^{L_{\mathfrak{a}_n}} .\] 
The computation rule (\ref{vivelesabeilles}) gives 
\[\hat{\Psi}_{\Omega_n,\mathfrak{a},\MFb}^{L_{\mathfrak{a}_n}}(\pi_n)=i_\MFp\left( \psi\left(\Omega;\MFa^{-1}\MFp^{n}L,\MFb^{-1}\MFa^{-1}\MFp^{n}L\right) \right).\]
\hfill $\square$ 
\end{proof}
In this last part of section \ref{colemanseries} we insert a lemma that will be used in the proof of Proposition \ref{lemchiuprimeyop}. Let $d\in\MCO_k$ be prime to $\MFf\MFp$. Then $\Omega':=d\Omega$ is also a primitive $\MFf$-torsion point of $\MBC/L$ satisfying the points 1. and 2. above. Moreover, the corresponding $u'_n:= \Lambda(\MFp^{-n},L)(1-x^n)\Omega'$ is nothing but $du_n$. Therefore
\[\theta_{\infty,E^{(n)}}(u_n')=d.\theta_{\infty,E^{(n)}}(u_n)=d.\theta_{\MFp,E^{(n)}}(\pi_n)=\theta_{\MFp,E^{(n)}}([d]_{E^{(n)}}(\pi_n)),\]
by the definition of $\pi_n$ and (\ref{morphism}). Thus if we define $\pi'_n$ by $\theta_{\MFp,E^{(n)}}(\pi'_n)=\theta_{\infty,E^{(n)}}(u_n')$ we see that
\begin{equation}\label{modificationbase}
\pi'_n=[d]_{E^{(n)}}(\pi_n).
\end{equation}
Since $[d]_{E^{(n)}}=\Phi^{-n}([d]_E)$ we obtain the following result
\begin{lm}\label{changementdeperiode} Let $d\in\MCO_k$ be prime to $6\MFf\MFp$. Let $\MFa$ be an ideal of $\MCO_k$, prime to $6\MFf\MFp$. Let $\MFb$ be an ideal of $\MCO_k$ prime to $6d\MFa\MFf\MFp$. Then
\[\hat{\Psi}^L_{\Omega,d\MFa,\MFb}=\hat{\Psi}^L_{d\Omega,\MFa,\MFb}\circ[d]_E.\]
\end{lm}
\begin{proof} Use the equality
\[\psi\left(\Omega;(d\MFa)^{-1}\MFp^{n}L,\MFb^{-1}(d\MFa)^{-1}\MFp^{n}L  \right)=\psi\left(d\Omega;\MFa^{-1}\MFp^{n}L,\MFb^{-1}\MFa^{-1}\MFp^{n}L  \right),\]
Proposition \ref{psicoleman} and (\ref{modificationbase}).\hfill$\square$
\end{proof}

\section{Logarithmic derivative of Coleman series}\label{logdercolser}

By the classical theory of formal groups we know that any translation-invariant derivation on $\widehat{E}(t,w)$ over $\mathcal{O}_{F_{\MFP}}$ (resp. $F_{\MFP}$) has the form $D_c=\frac{c}{\lambda'(t)}\frac{d}{dt}$, where $c\in\mathcal{O}_{F_{\MFP}}$ (resp. $c\in F_{\MFP}$). Let $D:=D_1$ then from (\ref{lambdaetchapeau}) we deduce that for any $g\in\overline{\MBQ}\left(\wp_L,\wp_L'\right)$ such that $\widehat{g}\in F_{\MFP}[[t]]$ we have 
\[D(\widehat{g})=\widehat{g'}.\]
Let $\MCU_{\MFf,n}(\MFP)$ be the pro-$p$-part of $\MCO_{\widetilde{F}_n}^\times$. Taking the projective limit under the norm maps, we define 
\[\MCU_{\MFf,\infty }(\MFP):=\varprojlim_n\MCU_{\MFf,n}(\MFP).\]
Recall that for $u\in\MCU_{\MFf,\infty }(\MFP)$, the Coleman power series $F_u$ satisfies $F_u(0)\equiv1$ modulo $\MFp_{F_{\MFP}}$ so that $\log\left( F_u \right)$ is well defined and has coefficients in $F_{\MFP}$. Moreover, if $F_u=\widehat{g}$, for some $g\in\overline{\MBQ}\left(\wp_L,\wp_L'\right)$ then

\[ D(\log\widehat{g})=\frac{\widehat{g'}}{\widehat{g}}.\]

To obtain $p$-adic $L$-functions we first construct $\mathcal{O}_{\mathbf D}$ valued measures, where ${\mathbf D}$ is the completion of $k_{\MFp}^{ur}$. Therefore, as Ehud de Shalit did in \cite[Lemma page 18]{deshalit87}, we have to consider the power series
\[\widetilde{\log}\,F_u:=\log F_u-\frac{1}{p}\sum_{\omega\in W^1_f}\log F_u(t[+]\omega)=\log F_u-\frac{1}{p}\log (F_u^{\Phi}\circ f),\]
where $t[+]\omega:=\widehat{E}(t,\omega)$. The second equality follows from the definition of the norm operator of Coleman given for intance in \cite[Proposition 2.1 page 11]{deshalit87}, usually denoted by $\mathcal{N}_f$ and the identity $\mathcal{N}_f(F_u)=F_u^{\Phi}$ proved in \cite[Corollary 2.3 page 14]{deshalit87}. The series $\widetilde{\log}\,F_u\in F_{\MFP}[[t]]$. But one immediately sees that  $\widetilde{\log}\,F_u\in\mathcal{O}_{F_{\MFP}}[[t]]$ thanks to the second equality and to the fact that $h^p\equiv h^{\Phi}\circ f$ modulo $\MFp_{F_{\MFP}}$ for any $h\in\mathcal{O}_{F_{\MFP}}[[t]]$. Since $\lambda^{\Phi}\circ f=\Lambda(\MFp,L)\lambda$ we see from above that if $F_u=\widehat{g}$ then
\[D(\widetilde{\log}\,\widehat{g})=\frac{\widehat{g'}}{\widehat{g}}-\frac{\Lambda(\MFp,L)}{p}\Bigl(\frac{\widehat{g'}^{\Phi}\circ f}{\widehat{g}^{\Phi}\circ f}\Bigr).\]

But since $h^p\equiv h^{\Phi}\circ f$ modulo $\MFp_{F_{\MFP}}$ for any $h\in\mathcal{O}_{F_{\MFP}}[[t]]$ we obtain
\begin{cor}\label{antecedent} Suppose $F_u=\widehat{g}$, for some $g\in\overline{\MBQ}\left(\wp_L,\wp_L'\right)$ then
\[D(\widetilde{\log}\,\widehat{g})\equiv\frac{\widehat{g'}}{\widehat{g}}-\frac{\Lambda(\MFp,L)}{p}\Bigl(\frac{\widehat{g'}}{\widehat{g}}\Bigr)^p\quad\mathit{modulo}\quad\MFp_{F_{\MFP}}[[t]].\]
\end{cor}

The field of rational functions on $E$ over $F$ is $F(x_\infty,y_\infty)$ with $x_\infty$ and $y_\infty$ satisfying (\ref{equacourbe}). Let $\Omega$, $\MFa$ and $\MFb$ be as in Corollary \ref{PsiOL} with $\MFm=\MFf$ and let $V:=\Psi^L_{\Omega,\MFa,\MFb}$. Then by using the equality $(x_\infty, y_\infty)=\theta_{\infty, E}(z)$ we see that
\[\mathcal{R}^L_{\Omega,\MFa,\MFb}:=\Bigr
(\frac{V'}{V}-\frac{\Lambda(\MFp,L)}{p}\Bigl(\frac{V'}{V}\Bigr)^p\Bigr)\circ\theta_{\infty, E}^{-1}\]
is a rational function on $E/F$. We draw the attention of the reader that it is necessary to take $\MFm=\MFf$ to be sure that $\mathcal{R}^L_{\Omega,\MFa,\MFb}$ is defined over $F$. Let $A$ be the localization of $\mathcal{O}_F$ at $\MFp_F:=\MFP\cap F$. Since $\widetilde{E}$ is assumed to be an irreducible curve $\MFp_F A[x_\infty,y_\infty]$ is a prime ideal of $A[x_\infty,y_\infty]$. Let $A[x_\infty,y_\infty]_\MFP$ be the localization of $A[x_\infty,y_\infty]$ at this ideal.
Then, by the proof of Proposition (\ref{PsiOL1}) the function $V\circ\theta_{\infty, E}^{-1}$ is actually a unit of $A[x_\infty,y_\infty]_\MFP$. Hence $\mathcal{R}^L_{\Omega,\MFa,\MFb}\in A[x_\infty,y_\infty]_\MFP$. Let $\MBF_q:=A/\MFp_FA=\MCO_F/\MFp_F$, $\overline{\MBF}_q$ be the algebraic closure of $\MBF_q$ and let
\[\red_1:A\left[x_\infty,y_\infty\right]_\MFP\longrightarrow \overline{\MBF}_q\left(x_\infty,y_\infty\right)\]
be the ring homomorphism extending the projection map $A\longrightarrow A/\MFp_FA$. Recall that $\overline{\MBF}_q\left(x_\infty,y_\infty\right)$ is the function field of $\widetilde{E}/\overline{\MBF}_q$.  We are interested in computing the polar divisor of $\red_1(\mathcal{R}^L_{\Omega,\MFa,\MFb})$. The divisor of $V$ on $\MBC/L$ is easily deduced from (\ref{diviseurpsi}) and (\ref{multiplicativity}). Moreover the poles of $\mathcal{R}^L_{\Omega,\MFa,\MFb}$ are all of order $p$ and come from the poles and the zeros of $V$. The maps $\theta_{\infty, E}$ and the reduction $E\longrightarrow\widetilde{E}$ give us a map $\widetilde{\theta}:kL/L\longrightarrow\widetilde{E}(\overline{\overline{\MBF}}_q)$ whose restriction to $(\MFa\MFb\MFf)^{-1}L/L$ is injective. Therefore the polar cycle of $\red_1\left(\mathcal{R}^L_{\Omega,\MFa,\MFb}\right)$ is
\[C^L_{\Omega,\MFa,\MFb}=p\Bigl[\sum_{r\in(\MFa\MFb)^{-1}L/L\atop r\not\in\MFa^{-1}L/L}\widetilde{\theta}(r-\Omega)+\varepsilon_{\MFb}\sum_{r\in\MFa^{-1}L/L}\widetilde{\theta}(r-\Omega)\Bigr],\]
where 
\begin{equation*}
\varepsilon_{\MFb}:=
\begin{cases}
0& \mathrm{if}\ p\vert(N(\MFb)-1),\\
1& \mathrm{otherwise}.
\end{cases}
\end{equation*}

Let us remark that reduction modulo the coefficients and the map $C_E$ defined in section \ref{Robert} give a commutative diagram

\begin{equation*}
\begin{CD}
A[x_\infty,y_\infty]_{\MFP}@>\red_1>> \overline{\MBF}_q(x_\infty,y_\infty)\\
@VC_EVV @Vc_EVV\\
\MCO_{F_{\MFP}}[[t]]_{\MFP}@>\red_2>>\overline{\MBF}_q((t)).  
\end{CD}
\end{equation*}
where the vertical maps are injectives and $c_E$ is deduced from $C_E$.
\begin{pr}\label{linind} Let $\Omega$ be as in \S \ref{colemanseries}. Suppose we have the follwing data.
\begin{enumerate}
\item $\MFa$ and $\MFb$ are proper ideals of $\MCO_k$ as in Proposition \ref{psicoleman}.
\item $\MCD\subset\MCO_k$ a set of nonzero elements all prime to $6\MFp\MFf\MFb$ such that the map $\MCD\longrightarrow\left(\MCO_k/\MFf\right)^\times$ is injective.
\end{enumerate}
Then the reductions $\red_2\left(D(\widetilde{\log}\,\hat{\Psi}^L_{d\Omega,\MFa,\MFb})\right)$, $d\in\MCD$, are linearly independant over $\overline{\MBF}_q$.
\end{pr}
\begin{proof}
By Proposition \ref{psicoleman} the power series $\hat{\Psi}^L_{d\Omega,\MFa,\MFb}$ is a Coleman power series $F_v$ for some $v\in\varprojlim_n\MCO^\times_{\widetilde{F}_n}$. Therefore Corollary \ref{antecedent} implies that
\[\red_2\left(D(\widetilde{\log}\,\hat{\Psi}^L_{d\Omega,\MFa,\MFb})\right)=\red_2\circ C_E\left(\mathcal{R}^L_{d\Omega,\MFa,\MFb}\right)=c_E\circ \red_1\left(\mathcal{R}^L_{d\Omega,\MFa,\MFb}\right).\]
Since $c_E$ is one to one we are reduced to consider $\red_1\left(\mathcal{R}^L_{d\Omega,\MFa,\MFb}\right)$. But these rational functions on $\widetilde{E}/\overline{\MBF}_q$ are linearly independant on $\overline{\MBF}_q$. For, the support of the polar cycles $C^L_{d\Omega,\MFa,\MFb}$ and $C^L_{d'\Omega,\MFa,\MFb}$ are disjoint unless $d=d'$.
\end{proof}

\section{The $p$-adic $L$ functions and their $\mu$-invariant}\label{fonctionLP}
Let us fix, once for all, primitive $p^n$ roots of unity $\zeta_n$ such that $\zeta_n^p=\zeta_{n-1}$. We recall that ${\mathbf D}$ is the completion of the maximal unramified extention of $k_{\MFp}$.

\subsection{The measure $\mu^{\MFf}$}\label{themeasure}
Let $\MFf$, $E$ and $L$ be as in section \ref{Robert} above.
Since the quotient $\Lambda(\MFp,L)/p$ is a unit in $F_{\MFP}\subset {\mathbf D}$, where $F=k(\MFf)$, there exists a unit $\Omega_p\in {\mathbf D}$ such that the following equality holds
\[\Omega_p^{\Phi-1}=\frac{\Lambda(\MFp,L)}{p},\]
whose proof may be found for instance in Iwasawa's book \cite[Lemma 3.11]{Iwasawa86}. Thus, if $\widehat{G}_m$ is the multiplicative formal group then there exists a unique isomorphism of formal groups $\eta:\widehat{G}_m\longrightarrow\widehat{E}$ such that $\eta\in X\MCO_{\mathbf D}[[X]]$, $\eta'(0)=\Omega_p$ and
\[f_E\circ\eta=\eta^{\Phi}\circ[p]_m,\] 
where  $[p]_m:=(1+X)^p-1$. This implies in particular that for any $a\in\MCO_{k_\MFp}=\MBZ_p$,
\begin{equation}\label{stability}
\eta^{-1}\circ[a]_E\circ\eta=[a]_m:=(1+X)^a-1.
\end{equation}
Let $\widetilde{F}_\infty:=\cup\widetilde{F}_n$ and let $\mathrm{G}:=\mathrm{Gal}(\widetilde{F}_\infty/F_{\MFP})$. Then, by \cite[Proposition 1.8 page 11]{deshalit87}, there exists a unique isomorphism of topological groups $\kappa: \mathrm{G}\longrightarrow\mathbb{Z}_p^{\times}$ such that $\kappa(\sigma)$ satisfies
\[[\kappa(\sigma)]_E(\omega)=\sigma(\omega),\quad\mathrm{for\ all}\ \omega\in W_{f_E}:=\cup W_{f_E}^n.\]
The map $\kappa$ does not dependant on our choice of $E$ or $\MFf$.  Indeed, on one hand the Galois group $G':=\mathrm{Gal}({\mathbf D}\widetilde{F}_\infty/{\mathbf D})$ is isomorphic to $G$. On the other hand, $\widetilde{F}_n{\mathbf D}={\mathbf D}(\zeta_n)={\mathbf D}(W_f^n)$ and $\kappa$ is just the action of $G'$ on the Tate module $\varprojlim_n<\zeta_n>$. Moreover the system
\[\omega_n:=\Phi^{-n}(\eta)(\zeta_n-1)\]
satisfy $\omega_n\in W_{\Phi^{-n}(f_E)}^n\backslash W_{\Phi^{-n}(f_E)}^{n-1}$ and $\Phi^{-n}(f_E)(\omega_n)=\omega_{n-1}$.\par
Let us now make our choice of $E$ and $L$ so that $L=\Omega\MFf$ for some $\Omega$, which in turn is automatically a primitive $\MFf$-torsion point of $\MBC/L$. Then by the general theory of formal groups, there exists $d\in\MBZ_p^\times$ such that $\pi_n=\Phi^{-n}([d]_E)(\omega_n)$, where $(\pi_n)_n$ is the system defined in Lemma \ref{baseTate}. Thus, replacing $\Omega_p$ by $d\Omega_p$ if necessary, and $\eta$ by $[d]_E\circ\eta$ we may suppose that $\omega_n=\pi_n$, for all $n$.\par

For any $\beta\in\MCU_{\MFf,\infty }(\MFP)$ the power series $g_\beta:=\widetilde{\log}\,F_{\Omega,\beta}\circ\eta\in\mathcal{O}_{\mathbf D}[[X]]$ gives rise to an $\mathcal{O}_{\mathbf D}$-valued measure $\mu_{g_\beta}$ on $\mathbb{Z}_p$ defined by
\[g_\beta(z)=\int_{\mathrm{G}}(1+z)^xd\mu_{g_\beta}(x),\quad\mathrm{for\ all}\ z\in\MFp_{\mathbf D}.\] 
Chapter 4 of Lang's book \cite{lang90}, particularly Theorems 1.1 and 1.2 are helpful to undertand the connection between measures on $\MBZ_p$ and power series. We should recall that if $H$ and $H'$ are any profinite abelian groups and $h:H\longrightarrow H'$ is any continuous map then we can use $h$ to carry a given $\MCO_{\mathbf D}$-valued measure $\mu$ on $H$ to obtain an $\MCO_{\mathbf D}$-valued measure $h_{*}\mu$ on $H'$ defined as follows. Let $\chi:H'\longrightarrow\MBC_p$ be a continuous map. Then
\[\int_{H'}\chi(\sigma)d(h_{*}\mu)(\sigma)=\int_{H}(\chi\circ h)(\tau)d\mu(\tau).\]
If $H=H'$ then the action of $\tau\in H$ on $\mu\in\Lambda(H,\MCO_{\mathbf D})$ is such that
\[\tau.\mu=(h_\tau)_*\mu,\]
where $h_\tau$ is multiplication by $\tau$. If $H=H'=\MBZ_p$ and $h=h_a$ for some $a\in\MBZ_p^\times$ then we have
\begin{equation}\label{rule}
(h_a)_*\mu_g=\mu_{g\circ[a]_m},
\end{equation}
where $\mu_g$ (resp. $\mu_{g\circ[a]_m}$) is the measure on $\MBZ_p$ associated to the power series $g$ (resp. $g\circ[a]_m$). Let us define the $\MCO_{\mathbf D}$-valued measure on $\mathrm{G}$
\[\mu_\beta:=\kappa^{-1}_*\mu_{g_\beta},\]
and let $\Lambda(\mathrm{G},\MCO_{\mathbf D})$ be the $\MBZ_p$-algebra of $\MCO_{\mathbf D}$-valued measures on $\mathrm{G}$. Since $F_{\Omega,\gamma(\beta)}=F_{\Omega,\beta}\circ[\kappa(\gamma)]$ for any $\gamma\in\mathrm{G}$ by \cite[Corollary 2.3 (iv) page 14]{deshalit87} we see that the map $\beta\longmapsto\mu_\beta$ is a $\MBZ_p[[\mathrm{G}]]$-homomorphism from $\MCU_{\MFf,\infty }(\MFP)$ into $\Lambda(\mathrm{G},\MCO_{\mathbf D})$. In particular
\begin{equation}\label{translation}
\mu_{\gamma(\beta)}(\gamma U)=\mu_\beta(U),
\end{equation} 
for any compact open subset $U$ of $\mathrm{G}$ and all $\gamma\in\mathrm{G}$.\par  
For any $n$ and any prime ideal $\MFD$ of $k(\MFf\MFp^\infty):=\cup k(\MFf\MFp^n)$ above $\MFp$ we denote $\MCU_{\MFf,n}(\MFD)$ the principal units in the completion of $k(\MFf\MFp^n)$ at $\MFD$.  We denote the projective limit $\varprojlim_n\MCU_{\MFf,n}(\MFD)$ by $\MCU_{\MFf,\infty}(\MFD)$ and we set
\[\MCU_{\MFf,\infty}:=\prod_{\MFD}\MCU_{\MFf,\infty}(\MFD),\]
the projective limit of semi-local units. Let us identify $\mathrm{G}$ to $\MCI_{\MFf}:=\mathrm{Gal}(k(\MFf\MFp^\infty)/k(\MFf))$ and let $i:\MCI_{\MFf}\longrightarrow G_{\MFf,\infty}:=\mathrm{Gal}(k(\MFf\MFp^\infty)/k)$ be the inclusion map. Then to any $\beta\in\MCU_{\MFf,\infty}$ we associate the following $\MCO_{\mathbf D}$-valued measure on $G_{\MFf,\infty}$,
\begin{equation}
\label{mesureglobale}
\mu^0_{\beta} := \sum_{\GGs\in \MCR} (h_{\sigma^{-1}}\circ i)_*\mu_{\left(\beta^{\GGs}\right)_\MFP},
\end{equation}
where $\MCR$ is any complete representative system of $G_{\MFf,\infty }$ modulo $\MCI_\MFf$ and, for any $\sigma\in\MCR$, $\left(\beta^{\GGs}\right)_\MFP$ is the canonical image of $\beta^\GGs$ in $\MCU_{\MFf,\infty }\left( \MFP \right)$. Let us remark that the definition of $\mu^0_{\beta}$ does not depend on $\MCR$. This is a consequence of (\ref{translation}). Moreover the map $\beta\longmapsto\mu^0_{\beta}$ is a $\MBZ_p[[G_{\MFf,\infty}]]$-homomorphism from $\MCU_{\MFf,\infty }$ into $\Lambda(G_{\MFf,\infty},\MCO_{\mathbf D})$.\par
For any $n\in\MBN$ we let $F_n:=k(\MFf\MFp^n)$ and map $\MCO_{F_n}^\times$ into the product $\prod_{\MFD}\MCU_{\MFf,n}(\MFD)$, where $\MFD$ describes the set of prime ideals of $F_n$ above $\MFp$, Let us denote $\upsilon_n$ this map,
\[\upsilon_n:\MCO_{F_n}^\times\longrightarrow\prod_{\MFD}\MCU_{\MFf,n}(\MFD).\]
Then we shall always assume that the component in $\MCU_{\MFf,n}(\MFP)$ of $\upsilon_n(\beta)$ is the projection of $i_{\MFp}(\beta)$. We are interested in the projective system of semi-local units 
\[\beta:=\upsilon_n(\psi\left( \Omega;\MFa^{-1}\MFp^{n}L,(\MFb\MFa)^{-1}\MFp^{n}L \right)),\]
where $\MFa$ and $\MFb$ are as in proposition \ref{psicoleman}. Since 
\[\psi\left( \Omega;\MFa^{-1}\MFp^{n}L,(\MFb\MFa)^{-1}\MFp^{n}L \right)=\psi\left( 1;\MFa^{-1}\MFf\MFp^{n},(\MFb\MFa)^{-1}\MFf\MFp^{n} \right),\]
the power series $g_{\left(\beta^{\GGs}\right)_\MFP}$ depends only on $\MFf$, $\MFa$, $\MFb$ and $\sigma$, and the measure $\mu^0_\beta$ depends only on $\MFf$, $\MFa$ and $\MFb$. We denote it $\mu^{\MFf}(\MFa,\MFb)$. If $\MFc$ is an ideal of $\MCO_k$ prime to $\MFf\MFp$ then we denote $\sigma_{\MFc}$ the automorphism $(\MFc, k(\MFf\MFp^\infty)/k)$ associated to $\MFc$ by the Artin map. The relations  
\begin{eqnarray*}
\psi\left( \Omega;\MFa^{-1}\MFp^{n}L,(\MFb\MFa)^{-1}\MFp^{n}L \right)&=&\psi\left( \Omega;\MFp^{n}L,\MFb^{-1}\MFp^{n}L \right)^{\sigma_{\MFa}}\\ 
\psi\left( \Omega;\MFp^{n}L,\MFb^{-1}\MFp^{n}L \right)^{N(\MFc)-\sigma_{\MFc}}&=&\psi\left( \Omega;\MFp^{n}L,\MFc^{-1}\MFp^{n}L \right)^{N(\MFb)-\sigma_{\MFb}} 
\end{eqnarray*}
imply 
\begin{equation*}
\mu^{\MFf}(\MFa,\MFb)=\sigma_{\MFa}\mu^{\MFf}(1,\MFb)\quad\mathrm{and}\quad(N(\MFc)-\sigma_{\MFc})\mu^{\MFf}(1,\MFb)=(N(\MFb)-\sigma_{\MFb})\mu^{\MFf}(1,\MFc).
\end{equation*}
As it is shown in \cite[proof of Theorem 4.12, pages 77-78]{deshalit87} (this is, in some how, also the conclusion of \cite[Proposition 273]{gillard85}) one may deduce from the second equality the existence of a unique $\MCO_{\mathbf D}$-valued measure $\mu^{\MFf}$ on $G_{\MFf,\infty}$ such that
\[\mu^{\MFf}(1,\MFb)=(N(\MFb)-\sigma_{\MFb})\mu^{\MFf}.\]
The norm relations (\ref{normpsi}) have a translation in terms of measures. Indeed, Let $\MFq$ be a prime ideal of $\MCO_k$ and let $\MFg=\MFf\MFq$. Then $\mu^{\MFg}$ induces by using the restriction map $res^{\MFg}_{\MFf}:G_{\MFg,\infty}\longrightarrow G_{\MFf,\infty}$ the following measure on $G_{\MFf,\infty}$ 
\begin{equation}\label{comportementmesure}
{res^{\MFg}_{\MFf}}_*\mu^{\MFg}=
\begin{cases}
\mu^{\MFf} & if\ \MFq\mid\MFf\\
(1-\sigma_{\MFq}^{-1})\mu^{\MFf} & if\ \MFq\nmid\MFf.
\end{cases}
\end{equation}
This allows us to define $\mu^{\MFf}$ even if $w_{\MFf}\neq1$. Indeed, suppose $\MFf\neq(1)$ and choose any positive integer $m$ such that $w_{\MFf^m}=1$. Then we set
\[\mu^{\MFf}:={res^{\MFf^m}_{\MFf}} _* \mu^{\MFf^m}.\]  
This definition does not depend on $m$ thanks to (\ref{comportementmesure}). If $\MFf=(1)$ then we remark that
\[(1-\sigma_{\MFq_2}^{-1}){res^{\MFq_1}_{(1)}}_*\mu^{\MFq_1}=(1-\sigma_{\MFq_1}^{-1}){res^{\MFq_2}_{(1)}}_*\mu^{\MFq_2},\]
for any prime ideals $\MFq_1\neq\MFp$ and $\MFq_2\neq\MFp$. Therefore there exists a unique pseudo-measure $\mu^{(1)}$ such that 
\[{res^{\MFq}_{(1)}}_*\mu^{\MFq}=(1-\sigma_{\MFq}^{-1})\mu^{(1)},\]
for any prime ideal $\MFq\neq\MFp$.

\subsection{The $p$-adic $L$-function $L_{p,\MFf}$ and associated power series}\label{Lfunction}
Let us fix 
\[\kappa_1:\mathrm{Gal}(k_\infty/k)\longrightarrow 1+p^\epsilon\mathbb{Z}_p\]
an isomorphism of topological groups, where $\epsilon:=1$ if $p\neq 2$ and $\epsilon:=2$ if $p=2$. Then for any $\MCO_{\mathbf D}$-valued measure $\mu$ on $G_{\MFf,\infty}$ and any $p$-adic character $\chi$ of $G_{\MFf,\infty}$ of finite order we define
\[L_{\mu,\MFf}(\chi, s):= L_{\mu,\MFf}(\chi\kappa_1^{-s}):=\int_{G_{\MFf,\infty}}\chi(\sigma)^{-1}\kappa_1(\sigma)^sd\mu(\sigma),\quad\mathrm{for\ all}\ s\in\MBZ_p,\]  
where $\kappa_1(\sigma):=\kappa_1(\sigma_{\vert k_\infty})$. We define $L_{\mu^{(1)},(1)}$ by the same formula when the pseudo-measure $\mu^{(1)}$ is considered, but then we assume $\chi\neq(1)$. By definition
\[L_{p,\MFf}:=L_{\mu^{\MFf},\MFf}.\]
Let us fix an isomorphism
\[G_{\MFf,\infty}\simeq\mathrm{Gal}(k_\infty/k)\times\mathrm{Gal}(k(\MFf\MFp^\infty)/k_\infty).\] 
Let $\gamma_0$ (resp. $c$) be a topological generator of $\mathrm{Gal}(k_\infty/k)$ (resp. $1+p^{\epsilon}\MBZ_p$) and let $\kappa_1$ be chosen so that $c:=\kappa_1(\gamma_0)$. Then any element of $1+p^\epsilon\mathbb{Z}_p$ is uniquely written as $c^x$, with $x\in\mathbb{Z}_p$. Thus we have a continuous map \[\ell:1+p^\epsilon\mathbb{Z}_p\longrightarrow\mathbb{Z}_p,\]
defined by $\ell(c^x)=x$. Let $\alpha:=\ell\circ\kappa_1$. We transform our measure $\mu$, or our pseudo-measure $\mu^{(1)}$, to obtain an $\MCO_{{\mathbf D}(\chi)}$-valued measure $\nu:=\alpha_*(\chi^{-1}\mu)$ on $\mathbb{Z}_p$, where ${\mathbf D}(\chi)$ is the finite extension of ${\mathbf D}$ generated by the values of $\chi$, supposed to be non-trivial when $\mu^{(1)}$ is considered. Let $G_\mu(\chi, X)\in\MCO_{{\mathbf D}(\chi)}[[X]]$ be the power series associated to this measure and let $\upsilon:=(\kappa_1)_*(\chi^{-1}\mu)$, then
\[L_{\mu,\MFf}(\chi, s)=\int_{1+p^\epsilon\mathbb{Z}_p}x^sd\upsilon(x)=\int_{\mathbb{Z}_p}c^{sx}d\nu(x)=G_\mu(\chi,c^s-1).\] 
See \cite[Chap.4, Theorem 1.2 and Example 2]{lang90}. Moreover, if we write $\chi=\chi_0\chi_1$, where $\chi_0$ is trivial on $\mathrm{Gal}(k_\infty/k)$ and $\chi_1$ is trivial on $\mathrm{Gal}(k(\MFf\MFp^\infty)/k_\infty)$, then
\[G_\mu(\chi, X)=G_\mu(\chi_0,\chi_1(\gamma_0)^{-1}(1+X)-1),\]  
whenever $\chi_0\neq1$ or $\mu\neq\mu^{(1)}$. The proof of such equality may be found in \cite[Chap.4, meas 2]{lang90}. By definition we set
\[G_{p,\MFf}(\chi,X):=G_{\mu^{\MFf}}(\chi, X).\]
We point out the relation
\[G_{p,\MFq}(\chi\circ res^{\MFq}_{(1)},X)=(1-\chi(\sigma_{\MFq})(1+X)^{-d})G_{p,(1)}(\chi,X),\quad d:=\alpha(\sigma_{\MFq}),\]
for any $\chi\neq(1)$ and any prime ideal $\MFq\neq\MFp$. Also we define
\[G_{p,(1)}(1,X):=G_{p,\MFq}(1,X),\]
where $\MFq\neq\MFp$ is any prime ideal satisfying $\gamma_0^{-1}=(\MFq, k_\infty/k)$. In particular, for $\chi=\chi_0\chi_1$ as above with $\chi_0=1$, we have
\[G_{p,(1)}(\chi_1,X)=\frac{G_{p,(1)}(1,\chi_1(\gamma_0)^{-1}(1+X)-1)}{1-\chi_1(\gamma_0)^{-1}(1+X)}.\]

Let $\MFf$, $E$ and $L$ be as in subsection \ref{themeasure}. In particular, we have $L=\Omega\MFf$. Let $\mu=\mu^0_\beta$ for some $\beta\in\mathcal{U}_{\MFf,\infty}$. Then one may write $G_\mu(\chi_0, X)$ in terms of what Gillard call the Iwasawa-Mellin-Leopoldt transform of measures obtained from $D_{\Omega_p}(\widetilde{\log}\,F_{\Omega,\left(\beta^{\GGs}\right)_\MFP})\circ\eta$, for $\sigma\in\mathcal{R}$. The differential operator $D_{\Omega_p}$ is already defined in the beginning of section \ref{logdercolser} as $\frac{\Omega_p}{\lambda'(t)}\frac{d}{dt}$. Here we should mention the relation $\lambda\circ\eta(X)=\Omega_p\log(1+X)$ from which we deduce
\[D_{\Omega_p}(\Psi)\circ\eta=\partial(\Psi\circ\eta),\]    
for any $\Psi\in F_{\MFP}[[t]]$, where $\partial=(1+X)\frac{d}{dX}$. By its very definition the measure $\mu_{\widetilde{\log}\,F_{\Omega,u}\circ\eta}$, for $u\in\MCU_{\MFf,\infty}(\MFP)$, has its support in $\MBZ_p^\times$. Therefore
\[\mu_{D_{\Omega_p}(\widetilde{\log}\,F_{\Omega,u})\circ\eta}=\mu_{\partial(\widetilde{\log}\,F_{\Omega,u}\circ\eta)}= \varphi\mu_{\widetilde{\log}\,F_{\Omega,u}\circ\eta}, \]
where $\varphi(x)=x$ for all $x\in\MBZ_p^\times$. The last equality is proved for instance in \cite[Chap.\,4 Meas 7 page 105]{lang90}. Let $\Delta_p$ be the group of roots of unity in $\MBZ_p^\times$. Then for any $\MCO_{\mathbf D}$-valued measure $\mu$ on $\MBZ_p$ we denote by $IML(\mu)$ the power series associated to the measure $\ell_*\ddot{\alpha}_{\vert1+p^\GVe\MBZ_p} $, where
\[\ddot{\alpha}:=\sum_{\GGd\in \Delta_p}(h_{\GGd})_* \mu.\]
If $\mu=\mu_g$ for some power series $g$ then we also write $IML(g)$ for $IML(\mu_g)$. In the proof of Proposition \ref{lemchiuprimeyop} below we shall use the relation
\begin{equation}\label{rule2}
(1+X)^dIML(g\circ\eta)=IML(g\circ[c^d]_E\circ\eta)
\end{equation}
easily deduced from the definition of the $IML$ transformation and from (\ref{stability}) and (\ref{rule}).
Let $f\in \MBN$ be such that $k(\MFf)\cap k_\infty=k(1)\cap k_\infty = k_f$ and let $\SCs=p^f$. Let $\gamma_0$ and $c$ be such that the restriction of $\kappa^{-1}(c)$ to $k_\infty$ is equal to      $\gamma_0^{\SCs}$,
\begin{equation}\label{choix}
\kappa^{-1}(c)\vert_{ k_\infty}=\gamma_0^{\SCs}
\end{equation}
Let $\omega:\MBZ_p^\times\longrightarrow\Delta_p$ be the group homomorphism such that $x=\GGw(x)c^{\ell(x)}$ and let us set $\GGw(x)=0$ for $x\in p\MBZ_p$. For any $u\in \MCU_{\MFf,\infty}\left( \MFP \right)$ we set 
\begin{equation}
\label{serieG}
\mu_u'(i):=\GGw^i\mu_{\partial(\widetilde{\log}\,F_{\Omega,u}\circ\eta)}\quad\mathrm{and}\quad G_u^i(X):=IML\left(\mu_u'(i)\right)(c^{-1}\left( 1+X \right)-1).
\end{equation}
\begin{lm}\label{liendeshalitgillard}
Let $\beta\in \MCU_{\MFf,\infty}$. Let $\chi=\chi_0$ be a $p$-adic character of $G_{\MFf,\infty}$ ($\chi_1=1$). Let $i\in\MBZ$ be such that $\chi^{-1}\circ\kappa^{-1}=\omega^i$. Then
\[G_{\mu^0_\beta}(\chi_0, X)
= \sum_{\GGs\in \MCR}\chi_0(\GGs)\left( 1+X \right)^{-<\GGs>}G^{i-1}_{\left(\beta^\GGs\right)_\MFP}\left( (1+X)^\SCs-1 \right),\]
where $<\GGs>:=(\ell\circ\kappa_1)(\sigma).$
\end{lm}
\begin{proof} Since the computations are straightforward we omit them. 
\end{proof}

\subsection{The $\mu$-invariant of $G_{\mu^0_\beta}(\chi_0, X)$}\label{muinvariant}
Let $\MFf$, $E$ and $L$ be as in section \ref{Robert} with the additional condition $L=\Omega\MFf$ for some complex number $\Omega$. Let $ord$ be the valuation of $\mathbb{C}_p$ normalized by $ord(p)=1$. Let $M$ be a finite extension of $\MBQ_p$. If $f=\sum a_nX^n$ is an element of $\MCO_{{\mathbf D}(M)}[[X]]$ associated to an $\MCO_{{\mathbf D}(M)}$-valued measure $\alpha$ on $\mathbb{Z}_p$, then by $\mu(f)$ or $\mu(\alpha)$ we mean the minimum of $ord(a_n)$, $n\in\MBN$. Usually we denote $f^*$ the power series associated to the restriction of $\alpha$ to $\MBZ_p^\times$. We have
\[f^*(X)=f(X)-\frac{1}{p}\sum_{\zeta^p=1}f(\zeta(1+X)-1).\]

\begin{theo}\label{gillardsinnott} (\cite[Theorem 1.5.1]{gillard85} for $p>3$)
Let $\MCF$ be a finite subset of $\MCO_{k_\MFp}^\times$ ($=\MBZ_p^\times$), such that for any $(a,b)\in \MCF^2$, 
if $\zeta a/b\in\MCO_k$, for some $\zeta\in\Delta_p$, then $a=b$. Let $M$ be a finite extension of $\MBQ_p$, and let $f_a$, $a\in\MCF$ be a set of power series in $\MCO_M[[t]]\cap M(\hat{x}_\infty,\hat{y}_\infty)$
\[\mu\Bigl(\sum_{a\in\MCF}IML\left(f_a\circ\eta{\circ}[a]_m\right)\Bigr)=\min_{a\in\MCF} \mu\Bigl(\sum_{\GGd\in \mu(k)}\left( f_a\circ\eta \right)^*\circ[\GGd]_m\Bigr)\]

\end{theo}
\begin{proof} If $p\geq5$ then this theorem is exactly Th\'eor\`eme 1.5.1 of Gillard's paper \cite{gillard85}. If $p=2$ or $p=3$ then Gillard's method still applies. Indeed, following a remark made by Sinnott in his proof of \cite[Theorem 1]{Sinnott84} and also used by Gillard, we may suppose that $(f_a\circ\eta)^*=f_a\circ\eta$ and $(f_a\circ\eta)\circ[\delta]_m=f_a\circ\eta$, for all $a\in\MCF$ and all $\delta\in\mu(k)$. In this case, and since $w_k=2$, we have to prove
\begin{equation}\label{formulepourmu}
\mu\Bigl(\sum_{a\in\MCF}IML\left(f_a\circ\eta\circ[a]_m\right)\Bigr)=2\min_{a\in\MCF} \mu\left( f_a\circ\eta\right).
\end{equation}
Let us denote $\alpha_a$, for $a\in\MCF$, the $\MCO_{{\mathbf D}(M)}$-valued measure of $\MBZ_p$ associated to $f_a\circ\eta$. Let $U:=1+p^\epsilon\MBZ_p$. Then it is easy to see that
\[({h_a}_{*}\alpha_a)\vert_U={h_a}_{*}({\alpha_a}\vert_{\delta_a U}),\quad {with}\quad\delta_a:=\omega(a)^{-1}.\]
Let $g_a\circ\eta$ be the power series associated to ${\alpha_a}\vert_{\delta_a U}$. Then, since $\Delta_p=\mu(k)=\{-1, 1\}$, the left hand side of (\ref{formulepourmu}) is equal to $2\mu\Bigl(\sum_{a\in\MCF}g_a\circ\eta\circ[a]_m\Bigr)$. Hence we have to show
\[\mu\Bigl(\sum_{a\in\MCF}g_a\circ\eta\circ[a]_m\Bigr)=\min_{a\in\MCF} \mu\left( f_a\circ\eta\right).\]
We may assume that $\min_{a\in\MCF} \mu\left( f_a\circ\eta\right)=0$. If $\mu\Bigl(\sum_{a\in\MCF}g_a\circ\eta\circ[a]_m\Bigr)>0$ then the relation
$\eta^{-1}\circ[a]_E\circ\eta=[a]_m$ implies that $\mu\Bigl(\sum_{a\in\MCF}g_a\circ[a]_E\Bigr)>0$.
As proved by Gillard in \cite[Proposition 2.3.3]{gillard87} there exists $M'$ a finite extension of $\MBQ_p$ such that $g_a\in\MCO_{M'}[[t]]\cap M'(\hat{x}_\infty,\hat{y}_\infty)$, for all $a$. This result is the analogue of \cite[Lemma 1.1]{Sinnott84}. Let us briefly explain how to check it. Fix $a\in\mathcal{F}$. Let $\chi_a$ be the characteristic function of $\delta_aU$ and decompose $\chi_a$ as a sum
\[\chi_a(u)=\frac{1}{p^{\epsilon}}\sum_{i=0}^{p^{\epsilon}-1}\theta^i\zeta_{\epsilon}^{iu},\]
where $\theta:=\zeta_\epsilon^{-\delta_a}$. Let $t:=\eta(X)$ and $t_i:=\eta(\zeta_\epsilon^i-1)$, then
\[g_a(t)=\frac{1}{p^{\epsilon}}\sum_{i=0}^{p^{\epsilon}-1}\theta^if_a(t[+]t_i),\]
where $t[+]t_i=\widehat{E}(t,t_i)$. The addition law on $E$(see for instance \cite[page 58]{Silverman92}) and the fact that the $t_i's$ are in $k(\MFf\MFp^\epsilon)_{\MFP}$ show that $\hat{x}_\infty(t[+]t_i)$ and $\hat{y}_\infty(t[+]t_i)$ are rational functions of $\hat{x}_\infty(t)$ and $\hat{y}_\infty(t)$ with coefficients in $\widetilde{ F}_\epsilon:=k(\MFf\MFp^\epsilon)_{\MFP}$. Hence if $f_a\in M(\hat{x}_\infty,\hat{y}_\infty)$ then $g_a\in M\widetilde{F}_\epsilon(\theta)(\hat{x}_\infty,\hat{y}_\infty)$.\par
By reducing $g_a$ modulo the maximal ideal of $M'$ we obtain the equality
\[\sum_{a\in\MCF}\bar{g}_a\circ[a]_{\widetilde{E}}=0.\]
By \cite[1.1.1 Th\'eor\`eme]{gillard85}, for any $a\in\MCF$ the power series $\bar{g}_a\circ[a]_{\widetilde{E}}$ is constant. Hence $\bar{g}_a$ is contant. Let $\pi$ be a uniformizer of $M'$ and $\upsilon_0$ be the Dirac measure at $0$. As Sinnott did in his case we translate the above conclusion into measures to obtain
\[{\alpha_a}\vert_{\delta_a U}=c_a\upsilon_0+\pi\alpha'_a,\]
for some constant $c_a\in\MCO_{{\mathbf D}(M')}$, and some $\MCO_{{\mathbf D}(M')}$ valued measure $\alpha_a'$ on $\MBZ_p$. But this is possible only if $c_a\equiv0$ modulo $\pi$. Since $\alpha_a$ is invariant by $[\delta]$ for $\delta\in\mu(k)$ we deduce that $\alpha_a\equiv0$ modulo $\pi$; which is a contradiction with the assumption  $\min_{a\in\MCF} \mu\left( f_a\circ\eta\right)=0$. This completes the proof of our theorem.
\hfill $\square$
\end{proof}

\medskip
Let $\MCJ$ be a set of ideals of $\MCO_k$, prime to $\MFf\MFp$, in bijection with $\Gal\left( k(1)/k \right)$ via the Artin map.
For each $j\in\{0,...,\SCs-1\}$, we define $\MCJ_j\subseteq\MCJ$ to be the subset formed by the $\MFc\in \MCJ$ such that 
$<\GGs_\MFc>\equiv j$ mod $\SCs$, where $\GGs_\MFc:=\left(\MFc, k(\MFf\MFp^\infty)/k\right)$.

\begin{lm}\label{MCFJ} If $\mathfrak{c}$ is a principal ideal of $\MCO_k$ prime to $\MFf\MFp$, generated by some $x\equiv1$ modulo $\MFf$ then
$\kappa(\sigma_{\MFc})=i_{\MFp}(x)$.
 
\end{lm}
\begin{proof} Since $\sigma_{\MFc}$ is the identity map on $F$ we deduce from section \ref{Robert} that the isogeny $\lambda_{\MFc}$ is simply the endomorphism of $E$ associated to $x$, and that $\sigma_{\MFc}$ commutes with $\theta_{\MFp, E}$. In particular for any $\omega\in W_{f_E}$ we have
\[\theta_{\MFp, E}(\sigma_{\MFc}(\omega))= \theta_{\MFp, E}(\omega)^{\sigma_{\MFc}}=x.\theta_{\MFp, E}(\omega).\]
But we know from (\ref{morphism}) that $x.\theta_{\MFp, E}(\omega)= \theta_{\MFp, E}([i_{\MFp}(x)]_E(\omega))$. This proves the lemma.
\hfill$\square$

\end{proof}
\begin{cor}\label{imagekappa1}(\cite[Proposition 2.3.7.]{gillard85} for $p>3$) If $\mathfrak{c}$ is a principal ideal of $\MCO_k$ prime to $\MFp$, generated by some $x\equiv1$ modulo $\MFp^\epsilon$ then
$\kappa_1(\tau_{\MFc})=i_{\MFp}(x)^{\SCs}$, where $\tau_\MFc:=\left(\MFc, k(\MFp^\infty)/k\right)$.
\end{cor}
\begin{proof} In his original proof Gillard used the grossencharacter of the elliptic curve $E$ over $k(\MFf)$. Here we give a more elementary proof using only lemma \ref{MCFJ}. Let $(x_n)_{n\geq\epsilon}$ be a sequence of elements in $\MCO_k$ all prime to $\MFf\MFp$, such that $x_n\equiv1$ modulo $\MFf$ and $x_n\equiv x$ modulo $\MFp^n$, for all $n\geq\epsilon$. In particular $i_{\MFp}(x_n)\in1+p^\epsilon\MBZ_p$. Let $\MFc_n:=x_n\MCO_k$, then by lemma \ref{MCFJ} we have $\kappa(\sigma_{\MFc_n})=i_{\MFp}(x_n)$ which implies, thanks to (\ref{choix}), that
\[\kappa_1(\sigma_{\MFc_n})=i_{\MFp}(x_n)^{\SCs}.\]
We deduce the corollary by using the continuity of $\kappa_1$.\hfill$\square$
\end{proof}
\begin{cor}\label{MCFJ}(\cite[Lemme 2.11.3.]{gillard85} for $p>3$)
For any $j\in\{0,...,\SCs-1\}$ the set $\MCF_j:=\lA c^{(j-<\GGs_\MFc>)/\SCs};\MFc\in \MCJ_j\rA$ satisfies the hypothesis of Theorem \ref{gillardsinnott}.
\end{cor}
\begin{proof} This is a consequence of the above corollary \ref{imagekappa1}. The proof of Gillard is valid even in the case $p\in\{2,3\}$. So we omit the details.\hfill$\square$
\end{proof}

\begin{pr}\label{lemchiuprimeyop} (\cite[Th\'eor\`eme 2.9.]{gillard85} for $p>3$) Let $\MFa$ and $\MFb$ be proper ideals of $\MCO_k$, as in Proposition \ref{psicoleman}. Let $\beta:=(\upsilon_n(u_n))$ be the projective system of semi-local units where
\[u_n= \psi\left( \Omega;\MFa^{-1}\MFp^{n}L,(\MFb\MFa)^{-1}\MFp^{n}L \right).\]
Then $G_{\mu^0_\beta}(\chi_0, X)$ is prime to $p$ in $\MCO_{{\mathbf D}(\chi_0)}[[X]]$.
\end{pr}
\begin{proof} Let $\MCR'$ be a complete system of representatives of $\Gal\left( k(\MFf\MFp^\infty) / k(1) \right)$ mod $\MCI_\MFf$.
Since $\mu_\beta^0$ does not depend on the choice of $\MCR$, we may take $\MCR=\sqcup_{j=0}^{\SCs-1}\MCR_j$, where $\MCR_j$ is the set of products $\GGs_\MFd\GGs'$ with $\MFd\in \MCJ_j$ and $\GGs'\in \MCR'$.
By Lemma \ref{liendeshalitgillard}, we have
\begin{equation*}
 G_{\mu^0_\beta}(\chi_0, X)= \sum_{j=0}^{\SCs-1} \left( 1+X \right)^{-j} S_j\left( (1+X)^\SCs-1 \right),
\end{equation*}
where for all $j\in\{0,...,\SCs-1\}$, 
\[S_j(X) = \sum_{\GGs\in \MCR_j} \chi_0(\GGs) \left( 1+X \right)^{(j-<\GGs>)/\SCs} G^{i-1}_{\left(\beta^{\GGs}\right)_{\MFP}}(X).\]
In such a situation the lemma 2.10.2 \cite{gillard85} is helpful, since by this lemma it is sufficient to show that $S_0$ is prime to $p$ or, equivalently, $S_0\left( c(1+X)-1 \right)$ is prime to $p$. If for $u\in\MCU_{\MFf,\infty}\left( \MFP \right)$ we denote $f_u$ the power series $D_{\Omega_p}(\widetilde{\log}\,F_{\Omega,u})$ then by (\ref{rule2}) and (\ref{serieG}) we have $S_0\left( c(1+X)-1 \right) = IML(\omega^{i-1}\mu_h)$, where
\begin{equation}
\label{expressSj}
h := \sum_{\GGs\in \MCR_0} \chi_0(\GGs) c^{-<\GGs>/\SCs} 
\left(f_{\left(\beta^{\GGs}\right)_{\MFP}}{\circ} \left[ c^{-<\GGs>/\SCs} \right]_E{\circ}\eta \right).
\end{equation} 
Also we may write
$h = \sum \chi_0\left( \GGs_\MFd \right) c^{-<\GGs_\MFd>/\SCs} h_\MFd\circ\left[c^{-<\GGs_\MFd>/\SCs}\right]_m$, 
where the sum is over all $\MFd\in\MCJ_0$ and where
\[h_\MFd := \sum_{\GGs\in \MCR'}  \chi_0(\GGs) c^{-<\GGs>/\SCs} 
\left(  f_{(\beta^{\GGs_\MFd\GGs})_\MFP}{\circ} \left[ c^{-<\GGs>/\SCs} \right]_E{\circ}\eta\right).\]
Let $\Xi\subset\MCO_k$ be a complete representative system of $(\MCO_k/\MFf)^\times$ modulo $\mu(k)$. We assume that all the elements $b$ in $\Xi$ are prime to $6\MFp\MFf$ and $b\equiv1$ modulo $\MFp^\epsilon$. Let us take $\MCR'$ to be the set of $\sigma_b:=(b\MCO_k, k(\MFf\MFp^n)/k)$ with $b\in\Xi$. Corollary \ref{imagekappa1} then implies
\[i_{\MFp}(b)=c^{<\sigma_b>/\SCs}.\]  
Since $F_\beta=\hat{\Psi}^L_{\Omega,\MFa,\MFb}$ we have
\[c^{-<\GGs_b>/\SCs} 
 f_{(\beta^{\GGs_\MFd\GGs_b})_\MFP}{\circ} \left[ c^{-<\GGs_b>/\SCs} \right]_E=b^{-1}D_{\Omega_p}(\widetilde{\log}\,\hat{\Psi}^L_{\Omega,b\MFa\MFd,\MFb})\circ[b^{-1}]_E=D_{\Omega_p}(\widetilde{\log}\,\hat{\Psi}^L_{b\Omega,\MFa\MFd,\MFb}).\]
The last equality being an application of the above lemma \ref{changementdeperiode}. At this stage we need the following remark. Let $\mu_f$ be an $\MCO_{{\mathbf D}(\chi_0)}$-valued measure on $\MBZ_p$ associated to some power series $f$ in $\MCO_{{\mathbf D}(\chi_0)}[[X]]$, and let $\omega^i*f\in\MCO_{{\mathbf D}(\chi_0)}[[X]]$ be the power series associated to the measure $\omega^i\mu_f$. Then it is easy to check that $f$ is prime to $p$ if, and only if, $\omega^i*f$ is prime to $p$. Thereby, Theorem \ref{gillardsinnott} and Lemma \ref{MCFJ} say us we only have to prove that for some $\MFd\in\MCJ_0$ the power series
\[\chi_0\left( \GGs_\MFd \right)c^{-\left<\GGs_\MFd\right>/\SCs}\Bigl(\sum_{\delta\in\mu(k)\atop b\in\Xi}\chi_0(\sigma_b)D_{\Omega_p}(\widetilde{\log}\,\hat{\Psi}^L_{b\Omega,\MFa\MFd,\MFb})\circ[\delta]_E\circ\eta\Bigr)\]
is prime to $p$. Since $D_{\Omega_p}(\widetilde{\log}\,\hat{\Psi}^L_{b\Omega,\MFa\MFd,\MFb})\circ[\delta]_E=\delta^{-1}D_{\Omega_p}(\widetilde{\log}\,\hat{\Psi}^L_{\delta^{-1}b\Omega,\MFa\MFd,\MFb})$, thanks to Lemma \ref{changementdeperiode}, we are led to investigate the coefficients of the power series
\[\sum_{\delta\in\mu(k)\atop b\in\Xi}\chi_0(\sigma_b)\delta^{-1}D_{\Omega_p}(\widetilde{\log}\,\hat{\Psi}^L_{\delta^{-1}b\Omega,\MFa\MFd,\MFb})\circ\eta.\]
The last step is to apply Proposition \ref{linind} whose second hypothesis is satisfied by the elements $\delta b$ of $\MCO_k$. Hence we deduce that the above sum has at least one coefficient prime to $\MFp_{{\mathbf D}(\chi_0)}$. This completes the proof of the proposition.\hfill$\square$
\end{proof}

\subsection{The $\mu$-invariant of $L_{p,\MFf}$}
The aim of this short subsection is to prove the following
\begin{theo}\label{gpfetlpf} For any $\MFf$ and any $p$-adic character $\chi=\chi_0$ of $G_{\MFf,\infty}$ ($\chi_1=1$), the power series $G_{p,\MFf}(\chi_0, X)$ is prime to $p$ in $\MCO_{{\mathbf D}(\chi_0)}[[X]]$.
\end{theo}
\begin{proof} The Theorem is a simple consequence of Proposition \ref{lemchiuprimeyop} since by construction we may assume that $w_{\MFf}=1$. But in this case we have

\[\Bigl(N(\MFb)-\chi(\sigma_{\MFb})^{-1}(1+X)^{<\sigma_{\MFb}>}\Bigr)G_{p,\MFf}(\chi,X)=G_\beta(\chi,X),\]
For any $\beta:=(\upsilon_n(\psi\left( \Omega;\MFa^{-1}\MFp^{n}L,(\MFb\MFa)^{-1}\MFp^{n}L \right)^{\nu}))$.
This concludes the proof.\hfill$\square$
\end{proof}

\section{From $p$-adic $L$ functions to $X_\infty$}\label{lien}
Let $K$ be the abelian extension of $k$ fixed in the introduction, and let us denote by $K_n$ (resp. $M_n$) the unique abelian extension of $k$ such that $K\subset K_n\subset K_\infty$ and $[K_n:K]=p^n$ (resp. the maximal abelian $p$-extension of $K_n$ unramified outside of $\mf{p}$. Since $M_\infty=\cup M_n$ the group $X_\infty$ is equal to the inverse limit

\begin{equation*}
X_\infty=\varprojlim_n\mr{Gal}(M_n/K_n),
\end{equation*}
Moreover, if $\Gamma_n:=\mathrm{Gal}(K_\infty/K_n)=\Gamma^{p^n}$ then it is easy to see that the module of co-invariants $(X_\infty)_{\Gamma_n}$ of $X_\infty$ is isomorphic to $\mr{Gal}(M_n/K_\infty)$,
\begin{equation*}
(X_\infty)_{\Gamma_n}\simeq\mr{Gal}(M_n/K_\infty).
\end{equation*}
But $\mr{Gal}(M_n/K_\infty)$ is finite. This is a consequence of class field theory and the non vanishing of the $\mf{p}$-adic regulator $R_{\mf{p}}(K_n)$, \cite[Theorem 2$'$]{Bru67}. Let us write $a\sim b$ for any numbers $a$ and $b$ such that the quotient $a/b$ is a $p$-unit.

\begin{lm} Let $e$ and $f$ be defined by $K\cap k_\infty=k_e$ and $k(1)\cap k_\infty=k_f$. Let $q:=4$ if $p=2$ and $q:=p$ otherwise. Then we have
\begin{equation}\label{cwgeneral}
[M_n:K_\infty]\sim qp^{n+e-f}\frac{h(K_n)R_{\mf{p}}(K_n)}{w_{K_n}\sqrt{\Delta_{\mf{p}}(K_n)}}\prod_{\mathfrak{P}}\Bigl(1-\frac{1}{N(\mathfrak{P})} \Bigr).
\end{equation}
As usual,  $\Delta_{\mf{p}}(K_n)$ is the $\mathfrak{p}$-component of the relative discriminant of $K_n/k$. The product in the formula is over all the prime ideals of $K_n$ lying above $\mathfrak{p}$.
\end{lm}

\begin{proof} This is \cite[Theorem 11]{CW76} in case $p>2$ and $p$ does not divide the class number of $k$. The general case is proved in a similar way. We point out that in case $p>2$ and $K=k(\mathfrak{fp})$ for some integral ideal $\mathfrak{f}$ of $k$ prime to $\mathfrak{p}$ this lemma is nothing but \cite[Proposition 2.7, page 112]{deshalit87}. \hfill$\square$
\end{proof}
\bigskip
Now let $\Theta_{I(\mathfrak{G})}$ be the group of elliptic units defined by R. Gillard and G. Robert in \cite[\S1]{GilRob79}. It is proved in \cite[Corollary of \S2]{GilRob79} that $\Theta_{I(\mathfrak{G})}$ is a subgroup of $\mathcal{O}_K^\times$ isomorphic as a Galois module to the augmentation ideal of $\mathbb{Z}[\mathrm{Gal(K/k)}]$. If $\chi$ is a non-trivial $\overline{\mathbb{Q}}_p$-character of $\mathrm{Gal}(K_n/k)$ then we define the sum $S(\chi)$ as follows
\begin{equation*}
S(\chi):=
\begin{cases}
\sum_{\sigma\in\mathrm{Gal}(k(\MFg_\chi)/k)}\chi(\sigma)\log(i_{\MFp}(\varphi_{\MFg_\chi}(\sigma))) & \mr{if}\  \MFg_\chi\neq(1),\\
 & \\
\frac{1}{h(k)}\sum_{\sigma\in\mathrm{Gal}(k(\MFg_\chi)/k)}\chi(\sigma)\log(i_{\MFp}(\delta(\sigma))) & \mr{if}\  \MFg_\chi=(1), 
\end{cases}
\end{equation*}
where $\MFg_\chi$ is the conductor of $\chi$. If $\MFg_\chi\neq(1)$ then $\varphi_{\MFg_\chi}(\sigma)\in k(\MFg_\chi)$ is the Robert invariant defined in \cite[\S 2.2 Définition page 15]{robert73}. See also \cite[formula (17) page 55]{deshalit87}. By $\delta(\sigma)\in k(1)$ we mean the Siegel unit defined for instance in \cite[chap.2 \S2]{Siegel61}. See also \cite[\S 3.1 Définition page 24]{robert73} or \cite[\S 2.2 Proposition page 49]{deshalit87}.\par
By using the formulas (2) and (3) in \cite{GilRob79} and the fact that
\begin{equation*}
[\mathcal{O}_K^\times:\Theta_{I(\mathfrak{G})}]\sim\frac{R_{\mf{p}}(\Theta_{I(\mathfrak{G})})w_{K_n}}{R_{\mf{p}}(K_n)},
\end{equation*}
we find the relation 
\begin{equation}\label{regulateur}
\frac{w_kR_{\mf{p}}(K_n)h(K_n)}{h(k)w_{K_n}}\sim\prod_{\chi\neq1}\frac{S(\chi)}{12g_\chi w_{\mathfrak{g}_\chi}},
\end{equation}
where the product is over all the non-trivial $\overline{\mathbb{Q}}_p$-characters of $\mathrm{Gal}(K_n/k)$. For each such $\chi$ we have written $g_\chi$ for the least positive integer in $\MFg_\chi$ and $w_{\mathfrak{g}_\chi}$ for the number of roots of unity in $k$ congruent to $1$ modulo $\MFg_\chi$.

From (\ref{cwgeneral}) and (\ref{regulateur}) we deduce
\begin{equation*}
[M_n:K_\infty]\sim\frac{qp^{n+e-f}h(k)}{w_k\sqrt{\Delta_{\mf{p}}(K_n)}}\prod_{\mathfrak{P}}\Bigl(1-\frac{1}{N(\mathfrak{P})} \Bigr)\prod_{\chi\neq1}\frac{S(\chi)}{12g_\chi w_{\mathfrak{g}_\chi}}.
\end{equation*}
But recall the $p$-adic Kronecker formula stated in \cite[Theorem 5.2 page 88]{deshalit87}. We have
\[L_{p,\MFf_\chi}(\chi,0)=-G(\chi^{-1})(1-\frac{\chi^{-1}(\MFp)}{p})\frac{S(\chi)}{12g_\chi w_{\mathfrak{g}_\chi}},\]
Where $\MFf_\chi$ is the part of $\MFg_\chi$ prime to $\MFp$. If $\MFg_\chi=\MFf_\chi\MFp^{n_\chi}$ and $\MFf$ is any multiple of $\MFf_\chi$ prime to $\MFp$ and such that $w_{\MFf}=1$ then 
\[G(\chi):=\frac{1}{p^{n_\chi}}\sum_{\gamma\in\mathrm{Gal}(k(\MFf\MFp^{n_\chi})/k(\MFf))}\chi(\gamma)\zeta_{n_\chi}^{-\kappa(\gamma)}.\]
It is easy to check the relation $G(\chi)G(\chi^{-1})=\chi(\tau)/p^{n_\chi}$ where $\tau$ is the restriction to $k(\MFf\MFp^{n_\chi})$ of $\kappa^{-1}(-1)$. Thus, by the conductor-discriminant theorem proved for instance in \cite[Theorem 7.15]{Iwasawa86} we obtain the relation
\[\Big(\prod_\chi G(\chi)\Big)^2=\frac{\pm1}{\Delta_{\mf{p}}(K_n)}.\] 
Further, we remark that $\prod_{\mathfrak{P}}(1-\frac{1}{N(\mathfrak{P})})=(1-1/p)\prod_{\chi\neq1}(1-\frac{\chi^{-1}(\MFp)}{p})$  and $(1-1/p)q/w_k\sim1$. Therefore
\[[M_n:K_\infty]\sim[k(1):k_f]p^{n+e}\prod_{\chi\neq1}L_{p,\MFf_\chi}(\chi,0).\]
Using the computations made in the first half of the subsection \ref{Lfunction} we have
\begin{equation*}
L_{p,\MFf_\chi}(\chi,0)=
\begin{cases}
G_{p,\MFf_\chi}(\chi_0,\chi_1(\gamma_0)^{-1}-1)& if\ \chi_0\neq1,\\
&\\
\frac{G_{p,(1)}(1,\chi_1(\gamma_0)^{-1}-1)}{1-\chi_1(\gamma_0)^{-1}}& if\ \chi_0=1.
\end{cases}
\end{equation*}
Since $\MFf_\chi=\MFf_{\chi_0}$ and the product of $\chi_1(\gamma_0)^{-1}-1$, for $\chi_1\neq1$ is equal to $p^{n+e}$ we obtain
\[[M_n:K_\infty]\sim[k(1):k_f]\prod G_{p,\MFf_{\chi_0}}(\chi_0,\zeta-1),\]
where the product concerns all the characters $\chi_0$ of $\mathrm{Gal}(K_n/k_{n+e})\simeq\mathrm{Gal}(K_\infty/k_\infty)$ and all the $p^{n+e}$-th roots of unity but $(\chi_0,\zeta)\neq(1,1)$. Now it is time to use Theorem \ref{gpfetlpf}, which says us that for a fixed $\chi_0$ there exists $n_0$ such that
\[G_{p,\MFf_{\chi_0}}(\chi_0,\zeta-1)\sim(\zeta-1)^{\lambda_{\chi_0}},\]
for all root of unity $\zeta$ of order $\geq p^{n_0}$, where $\lambda_{\chi_0}$ is the $\lambda$-invariant of $G_{p,\MFf_{\chi_0}}(\chi_0,X)$. Thus there is a constant $\nu_{an}$ such that
\begin{equation}\label{final}
[M_n:K_\infty]\sim[k(1):k_f]p^{\lambda_{an} n+\nu_{an}}
\end{equation}
for all sufficiently large $n$, with $\lambda_{an}:=\sum \lambda_{\chi_0}$. Since $[M_n:K_\infty]=\#(X_\infty)_{\Gamma_n}=p^{\mu_{\infty}p^n+\lambda_{\infty}n+\nu_{\infty}}$ for any $n>>0$, where $\mu_{\infty}$, $\lambda_{\infty}$ and $\nu_{\infty}$ are the Iwasawa invariants of $X_\infty$ we deduce that
\begin{equation}
\mu_\infty=0,\quad \lambda_{\infty}=\lambda_{an}\quad\mathrm{and}\quad\nu_{\infty}=\nu_{an}.
\end{equation}
This proves the theorem \ref{theomu} given in the introduction. The interested reader may take a look on \cite{viguie11c} where a comparison between the Iwasawa invariants of the projective limit of the $p$-class group and the projective limit of units modulo elliptic units.

\bibliographystyle{amsplain}

\noindent\texttt{hassan.oukhaba@univ-fcomte.fr}

\noindent\texttt{viguie@mathematik.uni-muenchen.de}
\end{document}